\documentclass[preprint,3p,leqno,backref=page]{elsarticle}

\usepackage{amssymb}
\usepackage{amsmath}
\usepackage[english]{babel}
\usepackage[utf8]{inputenc}
\usepackage{enumerate}
\usepackage{graphicx}
\usepackage{amsthm}

\usepackage{xcolor,paralist,hyperref,titlesec,fancyhdr,etoolbox}
\newtheorem{theorem}{Theorem}[section]

\newtheorem{lemma}[theorem]{Lemma}

\newtheorem{remark}[theorem]{Remark}

\usepackage{amsmath,amsthm,amsfonts,amssymb,dsfont}
\usepackage{bbm}
\usepackage{multicol}
\usepackage{hyperref}
\usepackage{empheq}

\usepackage{marginnote}

\renewcommand*{\backrefalt}[4]{
	\ifcase #1 %
	\else $\uparrow$ {\footnotesize #2.}%
	\fi
	}

\newcommand{\weak}{\stackrel{\ast}{\rightharpoonup}}

\usepackage{mathtools}

\usepackage{enumitem}
\setenumerate[0]{label=(\alph*)}

\usepackage{comment}

\hypersetup{
    colorlinks,
    citecolor=black,
    filecolor=black,
    linkcolor=black,
    urlcolor=black
}
\usepackage{stackengine}
\usepackage{multicol}

\newcommand{\innerh}[2]{\ensuremath{ {\left( #1 , #2 \right)} }}
\newcommand{\innerv}[2]{\ensuremath{ {\left\langle #1 , #2 \right\rangle} }}

\renewcommand{\div}{\mathrm{div}\,}
\newcommand{\A}{\mathcal{A}}

\newcommand{\lip}{\mathrm{Lip}}
\newcommand{\ro}{\mathcal{R}}

\def\R{\mathbb{R}}
\def\N{\mathbb{N}}

\def\D{\mathcal{D}}
\def\L{\mathcal{L}}

\def\FF{\mathcal{F}}
\def\GG{\mathcal{G}}

\let\rref=\ref
\let\reqref=\eqref

\journal{arXiv}

\begin{document}

\begin{frontmatter}



\title{Semilinear wave equations with time-dependent coefficients}                %% Article title

\author[label1]{Nenad Antonić}
\ead{nenad@math.hr}
\author[label1]{Matko Grbac}
\ead{matko.grbac@math.hr}

\affiliation[label1]{organization={Department of Mathematics, Faculty of Science, University of Zagreb},
             addressline={Bijenička cesta 30},
             city={Zagreb},
             postcode={10\thinspace000},
             country={Croatia}}

% \affiliation[label2]{organization={Department of Mathematics, Faculty of Science, University of Zagreb},
%           addressline={Bijenička cesta 30},
%             city={Zagreb},
%             postcode={10 000},
             %state={},
%             country={Croatia}}

%\author{Matko Grbac} %% Author name

%\affiliation{organization={Department of Mathematics, Faculty of Science, University of Zagreb},%Department and Organization
 %           addressline={Bijenička cesta 30}, 
  %          city={Zagreb},
   %         postcode={10 000}, 
    %        %state={},
     %       country={Croatia}}

%% Abstract
\begin{abstract}
We prove the existence of strong and weak solutions to the semilinear wave equation with coefficients depending both on time
and space variables, with continuous nonlinearity satisfying the sign condition.
The uniqueness is proven under slightly more restrictive assumptions.
Furthermore, the results obtained in abstract setting are illustrated on practical examples.
\end{abstract}

%% Keywords

\end{frontmatter}

\section{Introduction}

In nineteen-fifties the semilinear wave equation with cubic nonlinearity was considered the simplest model
of Lorentz invariant nonlinear interaction in quantum field theory \cite{IS63}.
The mathematical question of existence of global smooth solutions in three space dimensions was resolved
already by K.~Jörgens
in 1961 \cite{KJ61}. In the same paper an unpublished result by J.-L.~Lions on the existence of weak solutions
(under weaker assumptions) was mentioned, relying on the results in \cite{L61} (see also \cite{JLLiWS65}),
while the existence and uniqueness of strong (strict) solutions was established in \cite{FB62}.
A nice survey of results for semilinear equations with d'Alembertian $\Box=\partial_{tt}-\triangle$ as the second-order
operator can be found in \cite{WS89} (see also \cite{MS92}).
These results prompted a natural question of establishing a more general framework of operator-theoretical character,
which was already present in works of J.-L.~Lions and collaborators.

In the theory of linear wave equations, discontinuous variable coefficients both in space and time were already considered 
by C.~Baiocchi \cite{CB67}, following Lions' ideas (see also \cite{HS68} and \cite{dST74}).
It was shown that they could be taken to be functions of bounded variation in $t$.
In this paper we shall limit our interest only to absolutely continuous dependence, the more general
case to be investigated in a future paper (cf.~\cite{G24}).

The semilinear wave equation represents a fundamental model in mathematical physics and engineering,
capturing phenomena ranging from wave propagation in materials to the dynamics of fields in spacetime.
Understanding the behavior of solutions to such equations is crucial for various applications and
theoretical investigations. Our goal is to present existence and uniqueness theory for semilinear
wave equations with variable coefficients of low regularity.

There are numerous results for variable coefficients independent of time $t$, multiplying the space derivatives (i.e.~for the
second order operator of the form $\partial_t^2 - \textsf{div}(\textbf{A}\nabla \cdot)$).  In particular, if one additionally assumes that
the matrix $\textbf{A}$ is symmetric and strictly positive, the problem can be rephrased as a problem for equation with d'Alembertian on a
Riemannian manifold with metric $g$ having the components of $\textbf{A}^{-1}$ \cite{PFY11}, being a powerfull technique in various
control problems (for classical results on the control of semilinear wave equation with d'Alembertian see \cite{Z90}).

However, our main motivation for this study is preparation to extend the results of homogenisation (cf.~\cite{BOFM92, CD11}),
as well as propagation of microlocal energy density (\cite{AL01, FM92, PG96}) to the semilinear framework with variable coefficients.

In the Second section we first prove the existence of strong solutions, by assuming higher regularity in $t$, using the Galërkin approximations
and paying additional attention to time-variable coefficients. 
The main results are stated and proven in the next section, by approximation. Finally, in the last section we show
how this can be applied to the wave equation, even with non-symmetric coefficients (for the importance of such materials cf.~\cite{BM10}).

\bigskip
%\vfill

\subsection{Notation}

For a bounded open set $\Omega \subseteq \R^d$ with a Lipschitz boundary $\Gamma$ and $T>0$,
following a general approach advocated by J.-L.~Lions (cf.~\cite{, LM72, DL92}), we consider the initial-boundary value
problem for the abstract second order semilinear wave equation on $(0,T) \times \Omega$ (let us stress that $\ro$ and $\A$ depend on time variable $t$)
\begin{equation}\label{eq:problem}
    \left\{\begin{aligned}
        (\ro u')'+\A u + \FF(u) &= f+g\\
        u|_{[0,T] \times \Gamma}&=0\\
        u(0)&=u_0\\
        u'(0)&=u_1.
    \end{aligned}
    \right.
\end{equation}

For brevity, we shall denote $H=\textup{L}^2(\Omega)$ and $V=\textup{H}^1_0(\Omega)$, with $V'=\textup{H}^{-1}(\Omega)$
being the dual of $V$, so that $(V,H,V')$ form the \textit{Gel'fand triplet}\/ satisfying
\[
V \hookrightarrow H=H' \hookrightarrow V' ,
\]
with both inclusions being dense and compact. The inner product in $H$ will be denoted by $(\cdot,\cdot)$
and the duality pairing between $V'$ and $V$  by $\langle \cdot,\cdot \rangle$.

We keep the standard notation for Lebesgue/Sobolev/B\^ochner-Sobolev spaces (cf.~\cite{AF03, GL17}), but the shorter notation
for the corresponding norms will be used. For example, $\|\cdot\|_{\textup{W}^{k,p}(X)}$ instead of
$\|\cdot\|_{\textup{W}^{k,p}(0,T;X)}$; this will take place when there is no confusion over the time interval in the domain.

We shall occasionally also use the standard notation for inequalities between two functions $f,g$
\[
    f \lesssim g    \;,
\]
meaning that there is a constant $C>0$ such that $f \leq Cg$, and likewise
\[
    f \lesssim_p g \;,
\]
if the constant $C$ depends on some parameter (or a set of parameters) $p$. 

The nonlinear part of the equation arises from function $F : \R \to \R$, and is given by $\FF(u)(t,x)=F(u(t,x))$.
We shall also sometimes use $\FF(u(t))$ or $\FF(u)(t)$ to denote the function $x\mapsto F(u(t,x))$ for a fixed $t$.

%\vfill
%\eject

\section{Strong solutions}

We begin this section by making precise our assumptions for problem \reqref{eq:problem}.
In the first step we shall obtain higher regularity of a solution by assuming higher regularity in (abstract) coefficients.

Let $T>0$ be fixed. For operator family $\ro$ we assume that,   for some $\alpha > 0$
\begin{equation}\label{eq:ro_properties}
    \left\{\begin{aligned}
    &\ro \in \textup{W}^{2,1}(0,T;\L(H))&\\
    &(\ro(t)u,v)=(\ro(t)v,u),  &u,v \in H, t\in[0,T] \\
    &(\ro(t)u,u)\geq \alpha\|u\|_H^2, &u \in H, t\in[0,T]   \quad,
    \end{aligned}
    \right.
\end{equation}
while for $\A$ we assume to be of the form $\A=\A_0+\A_1$, where
\begin{equation}\label{eq:a0_properties}
    \left\{\begin{aligned}
    &\A_0 \in \textup{W}^{2,1}(0,T;\L(V;V')) &\\
    &\langle \A_0(t)u,v\rangle = \langle \A_0(t)v,u \rangle, &u,v\in V, t\in[0,T] \\
    &\langle \A_0(t)u,u \rangle \geq \alpha \|u\|^2_V, &u \in V, t\in[0,T]
    \end{aligned}
    \right.
\end{equation}
and
\begin{equation}\label{eq:a1_space}
    \A_1 \in \textup{W}^{1,1}(0,T;\L(V;H)).\\
\end{equation}
Furthermore, for the right-hand side of equation \reqref{eq:problem} we assume that $f \in \textup{W}^{1,1}(0,T;H)$
and $g \in \textup{W}^{2,1}(0,T;V')$.

Finally, for the non-linear part, we first start with a more restrictive assumption on $F$, namely that it is Lipschitz continuous
with the constant denoted by $\lip (F)$, i.e.
\begin{equation}\label{eq:lip}
    |F(z)-F(w)|\leq \lip(F)|z-w|, \quad z,w \in \R.
\end{equation}
We further assume that $F$ satisfies the \textit{sign condition}
\begin{equation}\label{eq:sign}
     z F(z) \ge 0, \qquad z \in \R,
\end{equation}
denoting its primitive function by
\[G(z)=\int_0^z F(w)dw,\]
so that $G'= F$ and $G(0)=0$.

\begin{remark}\label{rem:remark_on_coefficients}
\begin{enumerate}
    \item 
        Given that we have continuous inclusion $\textup{W}^{1,1}(0,T;X) \hookrightarrow \textup{C}([0,T];X)$  for each Banach space $X$, $(\rref{eq:ro_properties}_1)$ together with $(\rref{eq:a0_properties}_1)$ allows us to deduce that $\ro\in\textup{C}^1([0,T];\L(H))$,
        $\A_0\in\textup{C}^1([0,T];\L(V;V'))$ and $\A_1\in \textup{C}([0,T];\L(V;H))$.
        In the same vein we can get that $f \in C([0,T];H)$ and $g \in C^1([0,T];V')$.
    
    %     When there is enough regularity in functions $u \in H$ and $v \in V$, standard formulae for the (weak) derivative of products $\ro u$ and $\A_0 v$ hold, where we denote $(\ro u)(t)=\ro(t)u(t)$ (analogously for $\A_0$).
        
    %     Namely, if $u$ belongs additionally to $\textup{W}^{1,1}(0,T;H)$, (a) implies $\ro'u$ and $\ro u' \in \textup{L}^1(0,T;H)$ and consequently we have $\ro u \in \textup{W}^{1,1}(0,T;H)$ with
    %     \[
    %     (\ro u)' = \ro ' u + \ro u'.
    %     \]
    %     while for $v \in \textup{W}^{1,1}(0,T;V)$ we have $\A_0 v \in \textup{W}^{1,1}(0,T;V')$ with
    %     \[
    %     (\A_0 v)'=\A_0'v+\A_0 v'.
    %     \]
    %     Of course, higher order derivatives are obtained inductively when there is more regularity in both of the factors.
    
    \item 
        Symmetric operators $\ro(t)$, for $t \in [0,T]$, are bounded both from above and below, and therefore the same holds true
        for their inverses.
        From the following operator identity (which is valid in $\L(H)$)
        \[
        \ro(t+h)^{-1}-\ro(t)^{-1} = \ro(t+h)^{-1}\left[\ro(t)-\ro(t+h)\right]\ro(t)^{-1},
        \]
        (after denoting the map $t \mapsto \ro(t)^{-1}$ by $\ro^{-1}$) we see that $\ro^{-1}$ is continuous.
        Additionally, from the very same operator identity it also easily follows that
        \[
        \frac{d}{dt}\left(\ro(t)^{-1}\right)=-\ro^{-1}(t)\ro'(t)\ro^{-1}(t),
        \]
        so that $\ro^{-1} \in \textup{C}^1([0,T];\L(H))$.
        
    \item 
        Note that the sign condition \reqref{eq:sign} also implies
\begin{equation}\label{eq:g>0}
    G(z) \geq 0,\qquad z \in \R \;,
\end{equation}
    while after taking into account the continuity of $F$ we have that $F(0)=0$. Hence, by using \reqref{eq:lip} we can obtain the bound
\[
|F(z)|\leq \lip (F)|z|, \qquad z \in \R,
\]
    as well as
\begin{equation*}
G(z) \leq \int_0^{|z|} |F(w)|dw \quad \lesssim_{F}|z|^2, \qquad z \in \R.
\end{equation*}
    This in turn implies that
\begin{equation}
    \|\FF(u)\|_{H} \lesssim_F \|u\|_H, \qquad u \in H,
\end{equation}
    and
\begin{equation}\label{eq:G_l2_bound}
    \|\GG(u)\|_{\textup{L}^1(\Omega)} \lesssim_F \|u\|_H, \qquad u \in H.
\end{equation}
Of course, $\GG$ is defined from $G$ in the same way as $\FF$ from $F$.

       % \item If we denote $\A=\A_0+\A_1 \in \textup{W}^{1,1}(0,T;\L(V;V'))$, we can think of $\A_0$ as a symmetric and coercive part of an operator $\A$, while $\A_1$ represents a possible non-symmetric part.
    % \item Finally, let us mention a slight change of notation which will be used throught the first chapter. For brevity, we will occasionally ommit $0,T$ when writing corresponding Bochner spaces and respective components, and write $\textup{W}^{k,p}(X)$ instead of $\textup{W}^{k,p}(0,T;X)$.
\end{enumerate}   
\end{remark}

Now we are ready to prove the existence result given by the next theorem.

\begin{theorem}\label{thm:existence_prvi}
    Consider $\ro, \A_0$ and $\A_1$  satisfying \reqref{eq:ro_properties}, \reqref{eq:a0_properties} and  $\reqref{eq:a1_space}$ respectively,
    and let ${\cal F}$ satisfy both \reqref{eq:lip} and \reqref{eq:sign}.
    Furthermore, take $u^0,u^1 \in V$, $f\in \textup{W}^{1,1}(0,T;H)$ and $g \in \textup{W}^{2,1}(0,T;V')$, such that
    $\A_0(0)u^0-g(0) \in H$. Then there exists a unique solution
    \[
        u \in \textup{W}^{1,\infty}(0,T;V) \cap \textup{W}^{2,\infty}(0,T;H)
    \]
    with 
    \[
        \GG(u) \in \textup{L}^\infty(0,T;\textup{L}^1(\Omega)) \quad \text{and} \quad \A_0u-g \in \textup{L}^\infty(0,T;H)
    \]
    of the equation
    \begin{equation}\label{eq:thm}
        (\ro u')'+\A u+\FF(u)=f+g \qquad \text{ in } \textup{L}^2(0,T;V')
    \end{equation}
    with initial conditions
    \[u(0)=u^0, \qquad u'(0)=u^1.\]
\end{theorem}

%\bigskip

\begin{remark}
    \begin{enumerate}
        \item Note that the initial condition $u(0)=u^0$ makes sense for such a solution $u$; since $u \in \textup{W}^{1,\infty}(0,T;V)$, 
        it is also continuous, i.e.~$u \in \textup{C}([0,T];V)$.
        
        \item Since $u' \in \textup{W}^{1,\infty}(0,T;H)$, it follows that $u' \in \textup{C}([0,T];H)$ by the same argument.
        Hence the initial condition $u'(0)=u^1$ makes sense at least in the space $H$.
        %\item In fact, there is additional regularity (cf.~\cite[Theorem 9.3]{LM72}) which allows the interpretation
        %of equalities \dots
         % \item Given our assumptions on $\ro, \A_0,\A_1,\FF$, $f$ and $g$, the equation in \reqref{eq:problem} can be understood as equality in the sense of $\textup{L}^2(0,T;V')$. Indeed,
         % \[\|\FF(u)\|_{\textup{L}^2(0,T;H)} = \|\FF(u)-\FF(0)\|_{\textup{L}^2(0,T;H)} \leq \lip (F) \|u\|_{\textup{L}^2(0,T;H)}.\]
         % The conclusion then follows from the fact that \reqref{eq:thm} holds for each function in $\D(0,T)\boxtimes V$, the linear combinations of which are dense in $\textup{L}^2(0,T;V)$.
    \end{enumerate}
\end{remark}

\medskip

\subsection{Galërkin approximations}

In order to prove the existence of such a solution, we employ the Gal\"erkin method, with a carefully chosen basis.
We form an orthonormal basis $(w_k)$ for $H$ in the following way: first choose $w_1,w_2 \in V$ orthonormal such that
$\textup{span} \{u^0,u^1\} \leq \textup{span}\{w_1,w_2\}$, and then apply the Gram-Schmidt orthonormalisation procedure after adding
eigenvectors of $-\Delta$ on $H$ in order to obtain the desired orthonormal basis.

For any $m \in \N$, $m \geq 2$, denote $V_m=\text{span\,}\{w_1,\ldots,w_m\}$. Note that $u^0,u^1 \in V_m$, $m \geq 2$. We first show that there exists $u_m \in \textup{W}^{3,1}(0,T;V_m)$ which solves the problem projected to $V_m$

\begin{equation}\label{eq:proj}
\begin{cases}
\begin{aligned}
\innerh{\ro(t) u_m'(t)}{w_j} + \innerv{\A_0(t) u_m(t)}{w_j}\qquad\qquad\\
+\innerh{\A_1(t) u_m(t)}{w_j}+\innerh{\FF(u_m(t))}{w_j} &= \innerh{f(t)}{w_j}+\innerv{g(t)}{w_j},& \quad 1 \leq j \leq m\\
u_m(0) &= u^0 \\
u_m'(0)& = u^1. \\
\end{aligned}
\end{cases}
\end{equation}
We seek the solution of this projection in the form
\begin{equation}\label{eq:sol_form}
u_m(t)=\sum_{j=1}^m d_{mj}(t)w_j  .
\end{equation}
This leads to the following system of ordinary differential equations (for simplicity of notation, in the following 
computation we fix $m\in\N$ and do not write it explicitly as an index)
\begin{equation}\label{eq:ode}
    \frac{d}{dt}\biggl[\textbf{C}(t)\frac{d}{dt}\textbf{d}(t)\biggr]+\textbf{M}(t)\textbf{d}(t)+{\textbf{F}}(\textbf{d}(t))=\textbf{v}(t),
\end{equation}
where the above $m\times1$ and $m\times m$ matrices are given by their components
\begin{equation}
\begin{aligned}
\textbf{d}(t)     & =  [d_{j}](t)  \\
%\textbf{C}(t)_{ij}  & =  (\ro(t)w_j,w_i) \\
C_{ij}(t)           & =  (\ro(t)w_j,w_i) \\
%\textbf{M}(t)_{ij}  & =  \langle \A_0(t)w_j,w_i\rangle + (\A_1(t)w_j,w_i)\\
M_{ij}(t)           & =  \langle \A_0(t)w_j,w_i\rangle + (\A_1(t)w_j,w_i)\\
%\textbf{F}_j(a_1,\ldots,a_m)    & = \left(\FF\left(\sum_{i=1}^m a_iw_i\right),w_j \right)  \\
F_j(a_1,\ldots,a_m) & = \left(\FF\left(\sum_{i=1}^m a_iw_i\right),w_j \right)  \\
%\textbf{v}(t)_j & =  (f(t),w_j)+\langle g(t),w_j \rangle .
v_j(t)              & =  (f(t),w_j)+\langle g(t),w_j \rangle .
\end{aligned}
\end{equation}
Since $\textbf{C}(t)$ is a Gram matrix for orthogonal vectors $\sqrt{\ro(t)}w_j$, $j\in1..m$, it is invertible for each $t \in [0,T]$.
Moreover, by assumption $(\rref{eq:ro_properties}_1)$ on $\ro$ we have that $t \mapsto \textbf{C}(t)$ is in
$\textup{W}^{2,1}(0,T;\R^{m \times m})$, and in a similar fashion we deduce the same for its inverse $\textbf{C}(t)^{-1}$.
Next, using assumptions $(\rref{eq:a0_properties}_1)$ and $(\rref{eq:a1_space})$ on $\A_0$, $\A_1$ it can easily be seen that
$\textbf{M} \in \textup{W}^{1,1}(0,T;\R^{m\times m})$, and similarly that $\textbf{v} \in \textup{W}^{1,1}(0,T;\R^m)$.
Finally, since $F$ is Lipschitz, while $w_j$ are unit vectors in $H$, we have that $F_j$ are also Lipschitz.

%\[|\textbf{H}(\textbf{a}-\textbf{b})|=\left|\left(\FF((\textbf{a}-\textbf{b})\cdot \textbf{w}),\textbf{w} \right)\right| \leq\|\FF((\textbf{a}-\textbf{b})\cdot \textbf{w})\|_H\|\textbf{w}\|_H \lesssim_F |\textbf{a}-\textbf{b}|_\infty.\]

After introducing a new variable $\textbf{e}=\textbf{d}'$, \reqref{eq:ode} can be rewritten as an equivalent first order system of
ordinary differential equations
\begin{equation*}
    \begin{bmatrix}
        \textbf{I} & 0 \\ 0 & \textbf{C}
    \end{bmatrix}\frac{d}{dt}\begin{bmatrix} \textbf{d} \\ \textbf{e}\end{bmatrix}=
    \begin{bmatrix}
        0 & \textbf{I} \\ -\textbf{M} & -\textbf{C}'
    \end{bmatrix} \begin{bmatrix} \textbf{d} \\ \textbf{e}\end{bmatrix}+
    \begin{bmatrix}
        0 \\ \textbf{v}
    \end{bmatrix}- 
    \begin{bmatrix}
            0 \\ \textbf{F}(\textbf{d})
    \end{bmatrix},
\end{equation*}
which, due to the previous remark about invertibility of $\textbf{C}$ can further be written as
\begin{equation}\label{eq:ode_system}
    \frac{d}{dt}\begin{bmatrix} \textbf{d} \\ \textbf{e}\end{bmatrix}=
    \begin{bmatrix}
        \textbf{I} & 0 \\ 0 & \textbf{C}^{-1}
    \end{bmatrix}
    \left(
    \begin{bmatrix}
        0 & \textbf{I} \\ -\textbf{M} & -\textbf{C}'
    \end{bmatrix} \begin{bmatrix} \textbf{d} \\ \textbf{e}\end{bmatrix}+
    \begin{bmatrix}
        0 \\ \textbf{v}
    \end{bmatrix}- 
    \begin{bmatrix}
            0 \\ \textbf{F}(\textbf{d})
    \end{bmatrix}
    \right),
\end{equation}
subject to initial conditions
\begin{equation}
    \left\{\begin{aligned}
    d_j(0) & =  (u^0,w_j)\\
    e_j(0) & =  (u^1,w_j) \;.
    \end{aligned}
    \right.
\end{equation}
Since the right-hand side of system \reqref{eq:ode_system} is in $\textup{L}^2(0,T;\R^m)$, and it is Lipschitz continuous in $(\textbf{d},\textbf{e})$, the Carathéodory existence theorem hence yields a unique global absolutely continuous solution
$$
(\textbf{d}, \textbf{e}) \in \textup{W}^{1,1}(0,T;\R^m\times \R^m) \;.
$$
From the second equation in \reqref{eq:ode_system}
\[
\textbf{e}'=-\textbf{C}^{-1}\textbf{M}\textbf{d}-\textbf{C}^{-1}\textbf{C}'\textbf{e}
                                    +\textbf{C}^{-1}\textbf{v}-\textbf{C}^{-1}\textbf{F}(\textbf{d})
\]
we also deduce that $\textbf{e}' \in \textup{W}^{1,1}(0,T;\R^m)$. 
 % Since the product of two $\textup{W}^{1,1}(0,T)$ elements is once again $\textup{W}^{1,1}(0,T)$, we have that the first two factors, consisting of product of three matrices in $\textup{W}^{1,1}$, belong to $\textup{W}^{1,1}(0,T;\R^m)$ due to remarks that precede the introduction of the system \reqref{eq:ode_system}. For the last term, we have used the fact that $\textbf{H}$ is Lipschitz continuous (and hence differentiable almost everywhere) so that the composition $\textbf{H}(\textbf{d})$ satisfies $(\textbf{H}(\textbf{d}))'=\textbf{H}'(\textbf{d})\textbf{d}' \in \textup{L}^1(0,T;\R^m)$.
Hence, $\textbf{e} \in \textup{W}^{2,1}(0,T;\R^m)$, and subsequently, $\textbf{d} \in \textup{W}^{3,1}(0,T;\R^m)$.

Finally (now we reintroduce $m$ in our notation), this implies that $u_m$, being defined in \reqref{eq:sol_form},
is in $\textup{W}^{3,1}(0,T;V_m)$ and indeed a solution of \reqref{eq:proj}.

As a consequence, we also obtain 
\begin{equation}\label{eq:regularity_of_u_m}
\ro u_m' \in \textup{W}^{2,1}(0,T;H),\quad  \A_0 u_m \in \textup{W}^{2,1}(0,T;V') \quad\mbox{and}\quad  \A_1 u_m \in \textup{W}^{1,1}(0,T;H) \;.
\end{equation}

\medskip

\subsection{A priori estimates}\label{subsec:a_priori_estimates}

Next we shall obtain some a priori estimates for the sequence of solutions $(u_m)$. 
After multiplying $(\rref{eq:proj}_1)$ by $d_{mj}'(t)$ and summing in $j$ we obtain
\begin{equation}\label{eq:est1}
    ((\ro u_m')',u_m') + \langle \A_0 u_m,u_m' \rangle + (\A_1 u_m,u_m')+ (\FF(u_m),u_m')=(f,u_m')+\langle g,u_m'\rangle \;.
\end{equation}
We are going to use the following identities, which are valid because of the smoothness assumptions on $\ro$ and $\A_0$,
their symmetry $(\rref{eq:ro_properties}_2)$ and $(\rref{eq:a0_properties}_2)$, as well as \reqref{eq:regularity_of_u_m}:
\[
\begin{aligned}
( (\ro u_m')',u_m' ) & = \frac{1}{2}\Big( \frac{d}{dt} (\ro u_m',u_m') + (\ro'u_m',u_m') \Big) \nonumber \\
\langle \A_0 u_m,u_m'\rangle  & = \frac{1}{2}\Big(\frac{d}{dt}\langle \A_0u_m,u_m\rangle-\langle \A_0'u_m,u_m\rangle\Big) \;.\nonumber 
\end{aligned}
\]
After denoting the {\em energy}\/ at time $t$ by
\begin{equation}\label{eq:energy_m}
    E_m(t):=\frac{1}{2}(\ro(t) u_m'(t),u_m'(t))+\frac{1}{2}\langle \A_0(t)u_m(t),u_m(t)\rangle,
\end{equation}
equation \reqref{eq:est1} can be rewritten in the form
\begin{equation*}            %\label{eq:est2}
    E_m'(t) = -\frac{1}{2}(\ro'u_m',u_m')+\frac{1}{2}\langle \A_0'u_m,u_m\rangle-(\A_1u_m,u_m')-(\FF(u_m),u_m')+(f,u_m')+\langle g,u_m' \rangle \;.
\end{equation*}
We can now proceed with estimating the terms on the right hand side of the above equality: %\reqref{eq:est2}:
\begin{equation*}            % \label{eq:estimates_um_1}
    \begin{split}
    (\ro'(t)u_m'(t),u'_m(t)) &\leq \|\ro'(t)\|_{\L(H)} \|u_m'(t)\|_H^2 \quad\leq \quad \frac{\|\ro'(t)\|_{\L(H)}}{\alpha/2}E_m(t) \\[\jot]
    \langle \A_0'u_m(t),u_m(t)\rangle &\leq \|\A_0'(t)\|_{\L(V;V')}\|u_m(t)\|_V^2 \quad\leq \quad\frac{\|\A_0'(t)\|_{\L(V;V')}}{\alpha/2}E_m(t) \\[\jot]
    (\A_1(t)u_m,u_m'(t)) &\leq \|\A_1(t)\|_{\L(V;H)}(\|u_m'(t)\|_H^2+\|u_m(t)\|_V^2) \quad \leq \quad \frac{\|\A_1(t)\|_{\L(V;H)}}{\alpha/2}E_m(t) \\[\jot]
    (\FF(u_m(t)),u_m'(t)) &\leq \|\FF(u_m(t))\|_H\|u_m'(t)\|_H \quad \leq \quad  \frac{\lip(F)}{\alpha/2} E_m(t) \\[\jot]
    (f(t),u_m'(t)) &\leq \|f(t)\|_{H}(1+\|u_m'(t)\|_H^2) \quad \leq \quad \|f(t)\|_H+\frac{\|f(t)\|_H}{\alpha/2}E_m(t) \;,
    \end{split}
\end{equation*}
%
%From \reqref{eq:estimates_um_1} it follows
which gives that for each $0 \leq t \leq T$
\[E_m'(t) \leq \phi(t)E_m(t)+\|f(t)\|_H+\langle g(t),u_m'(t)\rangle,\]
where we have introduced an auxiliary function $\phi\in  \textup{L}^1(0,T)$ (let us note that it does not depend on solutions $u_m$)
\[
\phi(t)=\frac{2}{\alpha}\Bigl(\|\ro'(t)\|_{\L(H)}+\|\A_0'(t)\|_{\L(V;V')}+\|\A_1(t)\|_{\L(V;H)}+\|f(t)\|_H+\lip(F)\Bigr) \;.
\]
After integrating the previous estimate on $E_m'(t)$ from $0$ to $t$ we obtain
\[ 
E_m(t)  \leq  E_m(0)+\int_0^t \phi(s)E_m(s)ds+\int_0^t \|f(s)\|_H ds + \int_0^t      \langle g(s),u_m'(s)\rangle ds\Big.  \;.
\]
% First we estimate the term $\exp{\Big(\int_0^t \phi(s) \, ds\Big)}$, namely the exponent:
%     \begin{align*}
%        \int_0^t & \Bigl(\|\ro'(s)\|_{\L(H)}+\|\A_0'(s)\|_{\L(V;V')}+ \|\A_1(s)\|_{\L(V;H)}+\|f(s)\|_H\Bigr) \, ds\\[\jot]
%      \leq &\int_0^T \Bigl(\|\ro'(s)\|_{\L(H)}+\|\A_0'(s)\|_{\L(V;V')}+\|\A_1(s)\|_{\L(V;H)}+\|f(s)\|_H\Bigr) \, ds \\[\jot]
%      \leq &\|\ro'\|_{\textup{L}^1(\L(H))}+\|\A_0'\|_{\textup{L}^1(\L(V;V'))}+\|\A_1\|_{\textup{L}^1(\L(V;H))}+\|f\|_{\textup{L}^1(H)}\\[\jot]
%      \leq &\|\ro\|_{\textup{W}^{1,1}(\L(H))}+\|\A_0\|_{\textup{W}^{1,1}(\L(V;V'))}+\|\A_1\|_{\textup{L}^1(\L(V;H))}+\|f\|_{\textup{L}^1(H)}.
%     \end{align*}
% Hence, the first term is bounded by
% \[M:=\exp{\frac{2}{\alpha}\Big(\|\ro\|_{\textup{W}^{1,1}(\L(H))}+\|\A_0\|_{\textup{W}^{1,1}(\L(V;V'))}+\|\A_1\|_{\textup{L}^1(\L(V;H))}+\|f\|_{\textup{L}^1(H)}\Big)}.\]
% For the second factor we first have an obvious estimate
% \[\int_0^t \|f(s)\|_H ds \leq \|f\|_{\textup{L}^1(H)}.\]

In order to estimate the last term $\int_0^t \langle g,u_m' \rangle$, we apply the integration by parts formula and estimate the
obtained expression by using the elementary Young inequality with arbitrary (for the time being) constant $\kappa > 0$
\begin{equation}\label{eq:g_estimates_partial_integration}
    \begin{split}
     \int_0^t \langle g(s),u_m'(s)\rangle ds & = \langle g(t),u_m(t) \rangle - \langle g(0),u_m(0)\rangle - \int_0^t \langle g'(s), u_m(s) \rangle ds\\[\jot]
    &\leq 2\kappa \|g\|^2_{\textup{W}^{1,1}(V')}+\frac{1}{4\kappa}\left(\|u_m(t)\|^2_V+\|u_m(0)\|^2_V\right)    %\\[\jot]
%    &\quad
        + \int_0^t \|g'(s)\|_{V'}\|u_m(s)\|_V ds\\[\jot]
    &\leq 2\kappa \|g\|^2_{\textup{W}^{1,1}(V')} +\frac{1}{2\kappa\alpha}E_m(t)+\frac{1}{2\kappa\alpha}E_m(0)   %\\[\jot]
%    &\quad
        +\int_0^t \|g'(s)\|_{V'}(1+\|u_m(s)\|_V^2) ds\\[\jot]
    &\leq 2\kappa\|g\|^2_{\textup{W}^{1,1}(V')}+\|g'\|_{\textup{L}^1(V')} +\frac{1}{2\kappa\alpha}E_m(t)    %\\[\jot]
%    &\quad
        +\frac{1}{2\kappa\alpha}E_m(0)+\int_0^t \frac{\|g'(s)\|_{V'}}{\alpha/2}E_m(s)ds \,.
    \end{split}
\end{equation}
In this way we obtain the inequality
\begin{equation*}
    \begin{split}
        \left(1-\frac{1}{2\kappa\alpha}\right)E_m(t) & \leq \left(1+\frac{1}{2\kappa\alpha}\right)E_m(0)\\[\jot]
        &\qquad +2\kappa\|g\|_{\textup{W}^{1,1}(V')}^2+\|g'\|_{\textup{L}^1(V')}+\|f\|_{\textup{L}^1(H)}    %\\[\jot]
%        & \qquad
        + \int_0^t \tilde{\phi}(s)E_m(s)ds,
    \end{split}
\end{equation*}
where $\tilde{\phi}(s):=\phi(s)+\frac{2}{\alpha}\|g'(s)\|_{V'} \in \textup{L}^1(0,T)$ is also independent of solutions $u_m$.
At this point we can choose $\kappa$ such that $1 \geq \frac{1}{2\kappa\alpha} \geq \frac{1}{2}$. 

Finally, the initial term
\[E_m(0)=\frac{1}{2}(\ro u_m'(0),u_m'(0))+\frac{1}{2}\langle \A_0u_m(0),u_m(0) \rangle\]
can easily be estimated using the initial conditions
\[
    E_m(0) \leq \frac{2}{\alpha}\Bigl(\|\ro\|_{\textup{W}^{1,1}(\L(H))}\|u^1\|_H^2+\|\A_0\|_{\textup{W}^{1,1}(\L(V;V'))}\|u^0\|_V^2\Bigr) \;.
\]
By combining all the above pieces  we get the inequality
\begin{align*}
    E_m(t) \lesssim_{\alpha,\ro,\A,F,f,g}  \;1+\int_0^t \tilde{\phi}(s)E_m(s)ds.
\end{align*}

In order to obtain the sought a priori estimate,  we apply Gronwall's inequality once more, thus obtaining the inequality
\[
E_m(t) \lesssim_{\alpha,\ro,\A,F,f,g} \exp \left(\int_0^t \tilde{\phi}(s) ds\right) 
        \lesssim_{\alpha,\ro,\A,F,f,g} \exp \|\tilde{\phi}\|_{\textup{L}^1(0,T)} \;,
\]
which gives the desired uniform bound on $E_m(t)$. Consequently, we have obtained that 
\begin{equation*}
    \begin{split}
        &(u_m) \text{ is a bounded sequence in } \textup{L}^\infty(0,T;V)\;,\\
        &(u_m') \text{ is a bounded sequence in } \textup{L}^\infty(0,T;H)  \;.
    \end{split}
\end{equation*}

Our next goal is to obtain a priori estimates on the second time derivatives of sequence $(u_m)$.
To this end we return to the approximate equation
\begin{equation}\label{eq:solution_of_projected_for_each_v}
((\ro u_m')',v)+\langle \A_0u_m,v \rangle + (\A_1u_m,v)+(\FF(u_m),v)=(f,v)+\langle g,v \rangle \;,
\end{equation}
valid for every $v \in V_m$. We differentiate it with respect to $t$ (there is enough smoothness for this to be valid due to \reqref{eq:regularity_of_u_m}, as well as our initial assumptions) and insert $v=u_m''(t)$.
Thus, for each $0 \leq t \leq T$, we have got the equality
\begin{equation}\label{eq:estimates_um''}
    ((\ro u_m')'',u_m'')+\langle (\A_0u_m)',u_m'' \rangle + ((\A_1u_m)',u_m'')+((\FF(u_m))',u_m'')=(f',u_m'')+\langle g',u_m'' \rangle \;.
\end{equation}

The terms on the left hand side can be transformed in the following way
\begin{equation}
    \begin{split}
        ((\ro u_m')'',u_m'') & = (\ro''u_m',u_m'')+2(\ro'u_m'',u_m'')+(\ro u_m''',u_m'')\\[\jot]
        & = \frac{1}{2}\frac{d}{dt}(\ro u_m'',u_m'')+\frac{3}{2}(\ro'u_m'',u_m'')+(\ro''u_m',u_m'')\\[3\jot]
        \langle (\A_0 u_m)',u_m''\rangle & = \langle \A_0'u_m,u_m''\rangle+\langle \A_0u_m',u_m''\rangle\\[\jot]
        & = \frac{1}{2}\frac{d}{dt} \langle \A_0 u_m',u_m' \rangle -\frac{1}{2}\langle \A_0'u_m',u_m'\rangle+\langle \A_0'u_m,u_m''\rangle\\[3\jot]
        ((\A_1u_m)',u_m'') &= (\A_1 u_m',u_m'')+(\A_1'u_m,u_m'')\\[3\jot]
        ((\FF(u_m))',u_m'') &= (\FF'(u_m)u_m',u_m'') \;.
    \end{split}
\end{equation}
%\bigskip
After denoting (this expression is of the same form as energy $E_m$ in \reqref{eq:energy_m}, with $u_m$ replaced by $u_m'$)
\[
\Tilde{E}_m(t):=\frac{1}{2}(\ro(t) u_m''(t),u_m''(t))+\frac{1}{2}\langle \A_0(t) u_m'(t),u_m'(t) \rangle,
\]
we are able to rewrite \reqref{eq:estimates_um''} in the following expanded form
\begin{equation*}       %\label{eq:estimates_um''_2}
    \begin{split}
        \Tilde{E}_m'(t) &=\frac{d}{dt}\frac{1}{2}(\ro u_m'',u_m'')+\frac{d}{dt}\frac{1}{2}\langle \A_0 u_m',u_m' \rangle\\[\jot]
        &=((\ro u_m')'',u_m'')-\frac{3}{2}(\ro'u_m'',u_m'')-(\ro''u_m',u_m'')\\[\jot]
        &\quad + \langle (\A_0 u_m)',u_m''\rangle+\frac{1}{2}\langle \A_0'u_m',u_m'\rangle-\langle \A_0'u_m,u_m''\rangle\\[\jot]
        &=(f',u_m'')+\langle g',u_m''\rangle-((\A_1 u_m)',u_m'')-((\FF(u_m))',u_m'')\\[\jot]
        &\quad -\frac{3}{2}(\ro'u_m'',u_m'')-(\ro''u_m',u_m'')+\frac{1}{2}\langle \A_0'u_m',u_m'\rangle-\langle \A_0'u_m,u_m''\rangle\\[\jot]
        &=(f',u_m'')+\langle g',u_m''\rangle\\[\jot]
        & \quad - \frac{3}{2}(\ro'u_m'',u_m'')-(\ro''u_m',u_m'') +\frac{1}{2}\langle \A_0'u_m',u_m'\rangle-\langle \A_0'u_m,u_m''\rangle\\[\jot]
        &\quad -(\A_1 u_m',u_m'')-(\A_1'u_m,u_m'')-(\FF'(u_m)u_m',u_m'').
    \end{split}
\end{equation*}
What remains to be done is to estimate the terms appearing on the right hand side of    %\reqref{eq:estimates_um''_2}.
this equality. To this end, we estimate them one by one obtainig
\begin{equation*}           %\label{eq:estimates_um''_3}
    \begin{split}
        (f'(t),u_m''(t)) &\leq \|f'(t)\|_H\left( 1+\frac{1}{\alpha/2}\Tilde{E}_m(t)\right)\\[2\jot]
        (\ro'(t)u_m''(t),u_m''(t)) & \leq \frac{\|\ro'(t)\|_{\L(H)}}{\alpha/2}\Tilde{E}_m(t)\\[2\jot]
        (\ro''(t)u_m'(t),u_m''(t)) & \leq \|\ro''(t)\|_{\L(H)}\left(\|u_m'(t)\|_H^2+\|u_m''(t)\|_H^2\right)\\[2\jot]
        &\leq \frac{\|\ro''(t)\|_{\L(H)}}{\alpha/2}\Tilde{E}_m(t)\\[\jot]
        \langle \A_0(t)u_m'(t),u_m'(t) \rangle &\leq \frac{\|\A_0(t)\|_{\L(V;V')}}{\alpha/2}\Tilde{E}_m(t)\\[2\jot]
        (\A_1(t)u_m'(t),u_m''(t)) &\leq \frac{\|\A_1(t)\|_{\L(V;H)}}{\alpha/2}\Tilde{E}_m(t)\\[2\jot]
        (\A_1'(t)u_m(t),u_m''(t)) & \leq \|\A_1'(t)\|_{\L(V;H)}(\|u_m(t)\|_V^2+\|u_m''(t)\|_H^2)\\[\jot]
        & \leq \frac{\|\A_1'(t)\|_{\L(V;H)}}{\alpha/2}\left(E_m(t)+\Tilde{E}_m(t)\right)\\[2\jot]
        (\FF'(u_m(t))u_m'(t),u_m''(t)) & \leq \frac{\lip (F)}{\alpha/2}\Tilde{E}_m(t) \;.
    \end{split}
\end{equation*}
Combining these estimates with the equality and recalling the uniform bound on $E_m(t)$ obtained in the previous step, we reduce
%\reqref{eq:estimates_um''_2} to
it to the inequality
\begin{equation*}
    \begin{split}
        \Tilde{E}_m'(t) \leq \frac{1}{\alpha/2}\Tilde{\Tilde{\phi}}(t)\Tilde{E}_m(t)+\frac{\|E_m\|_{\textup{L}^\infty(0,T)}}{\alpha/2}\|\A_1'(t)\|_{\L(V;H)}+\|f'(t)\|_H+\langle g',u_m''\rangle - \langle \A_0'u_m,u_m''\rangle,
    \end{split}
\end{equation*}
where function $\Tilde{\Tilde{\phi}}\in \textup{L}^1(0,T)$ is given by
\begin{equation*}
%\begin{split}
    \Tilde{\Tilde{\phi}}(t):=\|\ro'(t)\|_{\L(H)}+\|\ro''(t)\|_{\L(H)}+\|\A_0(t)\|_{\L(V;V')}       %\\[2\jot]
    +\|\A_1(t)\|_{\L(V;H)}+\|\A_1'(t)\|_{\L(V;H)}+\lip (F) \;.
%\end{split}
\end{equation*}
Integrating the inequality from $0$ to $t\leq T$ we obtain
\begin{equation}            \label{eq:estimates_um''_reduced}
    \begin{split}
        \Tilde{E}_m(t) &\leq \Tilde{E}_m(0)+\frac{1}{\alpha/2}\int_0^t\Tilde{\phi}(s)\Tilde{E}_m(s)ds +\frac{\|E_m\|_{\textup{L}^\infty(0,T)}}{\alpha/2}\int_0^t \|\A_1'(s)\|_{\L(V;H)}ds\\[2\jot]
        &\quad +\int_0^t \|f'(s)\|_H ds+\int_0^t \langle g'(s),u_m''(s)\rangle ds-\int_0^t \langle \A_0'(s)u_m(s),u_m''(s)\rangle ds.
    \end{split}
\end{equation}
The last two terms are transformed via integration by parts:
\begin{equation}\label{eq:estimates_um''_0}
    \begin{split}
        \int_0^t \langle g'(s),u_m''(s)\rangle ds &= - \int_0^t \langle g''(s),u_m'(s)\rangle ds\\[2\jot]
        &\quad +\langle g'(t),u_m'(t) \rangle - \langle g'(0),u_m'(0) \rangle\\[2\jot]
        &\leq \int_0^t \|g''(s)\|_{V'}\left(1+\frac{1}{\alpha/2}\Tilde{E}_m(s)\right)ds\\[2\jot]
        &\quad + \frac{1+\alpha}{2\alpha}\|g'\|^2_{\textup{W}^{1,1}(V')}+\frac{1}{4}\Tilde{E}_m(t)+\frac{1}{2}\Tilde{E}_m(0),
    \end{split}
\end{equation}
and
\begin{equation}
    \begin{split}
        \int_0^t \langle \A_0'(s)u_m(s),u_m''(s)\rangle ds &= - \int_0^t \langle \A_0''(s)u_m(s),u_m'(s)\rangle ds-\int_0^t \langle \A_0(s)u_m'(s),u_m'(s)\rangle ds\\[2\jot]
        &\quad + \langle \A_0'(t)u_m(t),u_m'(t)\rangle -\langle \A_0'(0)u_m(0),u_m'(0)\rangle\\[2\jot]
        &\leq \frac{1}{\alpha/2}\int_0^t \|\A_0''(s)\|_{\L(V;V')}\left(E_m(s)+\Tilde{E}_m(s)\right)ds\\[2\jot]
        &\quad +\int_0^t \|\A_0\|_{\L(V;V')}\left(1+\frac{1}{\alpha/2}\Tilde{E}_m(s)\right)ds\\[2\jot]
        &\quad +\frac{1+\alpha}{2\alpha}\|\A_0'\|^2_{\textup{W}^{1,1}(\L(V;V'))}+\frac{1}{4}\Tilde{E}_m(t)+\frac{1}{2}\Tilde{E}_m(0)   \;.
    \end{split}
\end{equation}
Next we need an estimate of the term 
\begin{equation*}
   \begin{split}
       \Tilde{E}_m(0) &=\frac{1}{2}(\ro(0)u_m''(0),u_m''(0))+\frac{1}{2}\langle \A_0(0)u_m'(0),u_m'(0)\rangle\\[2\jot]
       &\leq \frac{1}{2}\|\ro\|_{\textup{W}^{1,1}(\L(H))}\|u_m''(0)\|_H^2+\frac{1}{2}\|\A_0\|_{\textup{W}^{1,1}(\L(V;V'))}\|u_m'(0)\|_V^2,
   \end{split} 
\end{equation*}
which obviously boils down to estimating the term $\|u_m''(0)\|_H$, since $u_m'(0)=u^1$ is trivially bounded in $V$.
In order to do so, we return back to the approximate equation \reqref{eq:solution_of_projected_for_each_v},
take $t=0$ and insert $v=u_m''(0)$, thus obtaining
\begin{equation}
    \begin{split}
        (\ro(0)u_m''(0),u_m''(0))&=(f(0),u_m''(0))+\langle g(0)-\A_0(0)u_m(0),u_m''(0)\rangle\\[2\jot]
        &\quad -(\ro'(0)u_m'(0),u_m''(0))- (\A_1(0)u_m(0),u_m''(0))-(\FF(u_m(0)),u_m''(0)).
    \end{split}
\end{equation}
At $t=0$ we have the following estimates (valid for some generic constant $C>0$)
\begin{equation}\label{eq:estimates_u_m''_4}
    \begin{split}
        (f(0),u_m''(0))&\leq C\|f\|_{\textup{W}^{1,1}(H)}^2+\frac{\alpha}{8}\|u_m''(0)\|_H^2\\[2\jot]
        (\ro'(0)u_m'(0),u_m''(0)) & \leq C\|\ro\|^2_{\textup{W}^{2,1}(\L(H))}\|u_m'(0)\|_V^2+\frac{\alpha}{8}\|u_m''(0)\|_H^2\\[2\jot]
        ( \A_1(0)u_m(0),u_m''(0) ) & \leq C\|\A_1\|^2_{\textup{W}^{1,1}(\L(V;H))}\|u_m(0)\|_V^2+\frac{\alpha}{8}\|u_m''(0)\|_H^2\\[2\jot]
        (\FF(u_m(0)),u_m''(0)) & \leq C\lip (F)^2\|u_m(0)\|^2_H+\frac{\alpha}{8}\|u_m''(0)\|_H^2    \;.
    \end{split}
\end{equation}
Recalling the initial condition ($\rref{eq:proj}_2$) for $u_m$, we easily get $\A_0(0) u_m(0)-g(0)=\A_0(0)u^0-g(0)$, which is
valid in $V'$. However, by the assumption of the theorem the right hand side is in $H$, so we also have
\begin{equation}\label{eq:estimates_u_m''_5}
    \langle g(0)-\A_0(0)u_m(0),u_m''(0) \rangle \leq C\|\A_0(0)u^0-g(0)\|_H^2+\frac{\alpha}{8}\|u_m''(0)\|_H^2.
\end{equation}
Collecting \reqref{eq:estimates_u_m''_4}--\reqref{eq:estimates_u_m''_5} we obtain the estimate
\begin{equation}\label{eq:estimates_u_m''_final}
    \begin{split}
        \|u_m''(0)\|_H^2 &\leq \frac{1}{\alpha}(\ro(0)u_m''(0),u_m''(0))\\[2\jot]
        & \leq C\Big(\|f\|_{\textup{W}^{1,1}(H)}^2+\|\ro\|^2_{\textup{W}^{2,1}(\L(H))}\|u^1\|_V^2+\|\A_0(0)u^0-g(0)\|_H^2\\[2\jot]
        &\quad +\|\A_1\|^2_{\textup{W}^{1,1}(\L(V;H))}\|u^0\|_V^2+\lip (F)^2\|u^0\|^2_H\Big)+\frac{5}{8}\|u_m''(0)\|_H^2,  
    \end{split}
\end{equation}
from where we deduce that
\begin{equation}\label{eq:u_m''_bounded}
    \|u_m''(0)\|_H \text{ is uniformly bounded}.
\end{equation}
This allows us to finally obtain the uniform bound on $\Tilde{E}_m$, via an application of the Gronwall inequality to \reqref{eq:estimates_um''_reduced} after taking into account estimates given by \reqref{eq:estimates_um''_0}--\reqref{eq:u_m''_bounded}.

\medskip

\subsection{The existence of a solution}

We are now in a position to prove the existence of a solution to \reqref{eq:thm}. Let $u_m$ be a solution to projected problem \reqref{eq:proj}. 
By using the above a priori estimates, we have that the sequence of solutions $(u_m)$ satisfies the following:
\begin{equation*}
\begin{split}
    (u_m) \text{ is bounded in }& \textup{L}^\infty(0,T;V)\\
    (u_m') \text{ is bounded in }& \textup{L}^\infty(0,T;V)\\
    (u_m'') \text{ is bounded in }& \textup{L}^\infty(0,T;H).
\end{split}
\end{equation*}

Therefore we can extract a subsequence (which we still denote by $(u_m)$) satisfying (as the derivative is 
a continuous operator on $\D'$) 
\begin{equation}\label{eq:cvg_projected_solutions}
\begin{aligned}
    u_m & \stackrel{\ast}{\rightharpoonup}  u \quad &\text{ in } \textup{L}^\infty (0,T;V)&\\
    u_m' & \stackrel{\ast}{\rightharpoonup}  u' \quad &\text{ in } \textup{L}^\infty(0,T;V)&\\
    u_m'' &\weak u'' \quad &\text{ in } \textup{L}^\infty(0,T;H)&.
\end{aligned}
\end{equation}
Additionally, we have
\begin{equation}\label{eq:cvg_linear_part}
\begin{split}
    \ro u_m' &\xrightharpoonup{\ast} \ro u' \quad \text{ in } \textup{L}^\infty(0,T;H) \\[2\jot]
    (\ro u_m')' &\weak (\ro u')' \quad \text{ in } \textup{L}^\infty(0,T;H)  \\[2\jot]
    \A_0 u_m &\stackrel{\ast}{\rightharpoonup} \A_0 u \quad \text{ in } \textup{L}^\infty(0,T;V') \\
    \A_1 u_m &\stackrel{\ast}{\rightharpoonup}  \A_1 u \quad \text{ in } \textup{L}^\infty(0,T;H).
\end{split}
\end{equation}
Lastly, we examine the convergence of nonlinear part $\FF(u_m)$. 
By convergence in $(\ref{eq:cvg_projected_solutions}_1)$, which implies the convergence in $\textup{L}^\infty(0,T;H)$,
uniformly for $t \in [0,T]$ up to a set of measure zero we have
\[\|\FF(u_m(t))-\FF(u(t))\|_H  \leq \lip(\FF) \|u_m(t)-u(t)\|_{H} \longrightarrow 0\:,\]
which gives
\begin{equation}\label{eq:F_conv}
    \FF(u_m) \longrightarrow \FF(u) \qquad \text{ in } \textup{L}^\infty(0,T;H).
\end{equation}

Take $n \in \N$ and $v_n \in V_n$, and let $m \geq n$. Since $u_m$ is a solution of projected problem \reqref{eq:proj} on $V_m$, which contains $V_n$, we have for each $t \in [0,T]$
\begin{equation*}       %\label{eq:solution_m_on_V_n}
    ((\ro(t) u_m'(t))',v_n)+\langle \A_0(t) u_m(t),v_n \rangle
    +(\A_1 (t)u_m(t),v_n)+(\FF(u_m(t)),v_n)=(f(t),v_n)+\langle g(t),v_n \rangle.        %\nonumber
\end{equation*}
Multiplying the above equality  %\reqref{eq:solution_m_on_V_n} 
by a cut-off function $\vartheta \in \D(0,T)$ and integrating over $[0,T]$ we obtain for $\psi_n := \vartheta \boxtimes v_n$
\begin{equation}\label{eq:solution_m}
    \int_0^T \langle (\ro u_m')', \psi_n\rangle + \langle \A_0 u_m,\psi_n \rangle + (\A_1u_m,\psi_n)+ (\FF(u_m),\psi_n)\, dt = \int_0^T (f,\psi_n)+\langle g,\psi_n \rangle dt.
\end{equation}
We then pass to the limit $m \to \infty$, while using \reqref{eq:cvg_projected_solutions}, \reqref{eq:cvg_linear_part} and \reqref{eq:F_conv} to obtain 
\begin{equation}\label{eq:solution}
    \int_0^T ((\ro u')',\psi_n) + \langle \A_0 u,\psi_n \rangle + (\A_1u,\psi_n) + (\FF(u),\psi_n)\, dt = \int_0^T (f,\psi_n)+\langle g,\psi_n \rangle dt.
\end{equation}
Since $\text{span}\{ \D(0,T)\boxtimes \bigcup_n V_n\}$ is dense in $\textup{L}^2(0,T;V)$, we can now deduce that
\begin{equation*}
    \int_0^T -(\ro u',\psi') + \langle \A_0 u,\psi \rangle + (\A_1u,\psi) + (\FF(u),\psi)\, dt = \int_0^T (f,\psi)+\langle g,\psi \rangle dt
\end{equation*}
holds for each $\psi \in \textup{L}^2(0,T;V)$, or in other words, we have that the equality
\[
(\ro u')'+\A u+\FF(u)=f+g
\]
holds in the sense of $\D'(0,T;V')$. From here, due to the density of smooth vector functions from $\D(0,T;V)$ in $\textup{L}^2(0,T;V)$,
which is a predual of $\textup{L}^2(0,T;V')$, we deduce \reqref{eq:thm}.

The fact that $\GG(u) \in \textup{L}^\infty(0,T;\textup{L}^1(\Omega))$ follows from \reqref{eq:G_l2_bound} and \reqref{eq:cvg_projected_solutions}, providing the estimate
\[\|\GG(u(t))\|_H \lesssim_F\|u(t)\|_H^2 \lesssim 1, \qquad t \in [0,T].\]
Furthermore, directly from the equation we obtain
\[
    \A_0u-g=f+g-(\ro u')'-\A_1u-\FF(u) \in \textup{L}^\infty(0,T;H).
\]

It remains to be shown that $u$ satisfies the initial conditions. First we check that $u'(0)=u^1$.

Take a cut-off function $\varphi \in \textup{C}^\infty([0,T])$ such that $\varphi(0)=1$ and $\varphi(T)=0$ and take $v \in V_m$ for some $m$.
Denote their tensor product by $\psi=\varphi \boxtimes v \in \textup{C}^\infty([0,T];V)$, and insert it into \reqref{eq:solution},
thus obtaining
\begin{equation*}\label{eq:iv_1}
    \int_0^T ( (\ro u')',\psi ) + \langle \A_0 u,\psi \rangle+(\A_1u,\psi)+(\FF(u),\psi) \, dt = \int_0^T (f,\psi)dt.
\end{equation*}
After integrating by parts in the first term it follows
\begin{equation*}\label{eq:iv_2}
    \int_0^T -(\ro u',\psi') + \langle \A_0 u,\psi \rangle+(\A_1u,\psi)+(\FF(u),\psi) \, dt + (\ro u'(0),\psi(0))  = \int_0^T (f,\psi)dt.
\end{equation*}
On the other hand, by inserting $\psi$ into \reqref{eq:solution_m} and once again integrating by parts we get
\begin{equation}\label{eq:iv_3}
    \int_0^T -((\ro u_m'),\psi') + \langle \A_0 u_m,\psi \rangle+ (\A_1u_m,\psi)+(\FF(u_m),\psi) \, dt + (\ro u_m'(0),\psi(0)) = \int_0^T (f,\psi)dt.
\end{equation}
Now recall that $u_m'(0) = u^1$, and take into account  \reqref{eq:cvg_linear_part} and \reqref{eq:F_conv},
in order to obtain
\begin{equation}\label{eq:iv_4}
     \int_0^T (\ro u',\psi') + \langle \A_0 u,\psi \rangle+(\A_1 u, \psi)+(\FF(u),\psi) \, dt + (\ro(0)u^1,\psi(0))  = \int_0^T (f,\psi)dt.
\end{equation}
Comparing \reqref{eq:iv_3} and \reqref{eq:iv_4} we deduce
\[(\ro u'(0),v)=(\ro(0)u^1,v).\]
Since $m$ and $v \in V_m$ were arbitrary, we conclude that
$$\ro(0) (u'(0)-u^1)=0,$$
and since $\ro(0)$ is an isomorphism of $H$, we conclude $u'(0)=u^1$.
In order to prove $u(0)=u^0$, we additionally assume $\varphi'(0)=1$ and $\varphi'(T)=0$ and integrate by parts once more in first terms of both \reqref{eq:iv_3} and \reqref{eq:iv_4}, while using symmetricity of $\ro$ to obtain
\begin{equation*}
    \int_0^T (\ro u',\psi')dt = -\int_0^T (u,(\ro \psi')')dt - (u(0),\ro\psi (0)),
\end{equation*}
\begin{equation*}
    \int_0^T (\ro u_m',\psi')dt = -\int_0^T (u_m,(\ro \psi')')dt - (u_m(0),\ro\psi (0)).
\end{equation*}
Recalling that $u_m(0) = u^0$ and $\reqref{eq:cvg_linear_part}_1$ we get by comparison
\[(u(0),\ro(0)v)=(u^0,\ro(0)v).\]
Once again, using the fact that $\ro(0)$ is an isomorphism of $H$, as well as that $m, v$ were arbitrary, we finally conclude
\[u(0)=u^0.\]

\bigskip
Finally, we prove the uniqueness of the solution. Assume $u_1$ and $u_2$ are two solutions of the problem, and denote $u:=u_1-u_2$. 
We wish to prove that $u \equiv 0$ is the only solution to the problem
\[\begin{cases}
\begin{aligned}
(\ro u')'+ \A u+\FF(u_1)-\FF(u_2)&=0,\\
u(0)=u'(0)&=0,
\end{aligned}
\end{cases}\]
Since $u' \in \textup{L}^2(0,T;V)$, we can use it as a test function for the equation and obtain
\begin{equation}\label{eq:unique_eq}
    ((\ro u')',u')+\langle \A_0 u,u'\rangle + (\A_1 u,u') + (\FF(u_1)-\FF(u_2),u')=0.
\end{equation}
The nonlinear term satisfies
\begin{align*}
    (\FF(u_1)-\FF(u_2),u') &\leq \lip(\FF) \|u_1(t)-u_2(t)\|_{H}\|u'(t)\|_H\\[2\jot]
    &\leq\lip(\FF)\big(\|u(t)\|_\textup{V}^2+\|u'(t)\|_H^2\big) \:,
\end{align*}
thus allowing us to rewrite \reqref{eq:unique_eq} as
\begin{align*}
    \frac{1}{2}\frac{d}{dt}\left((\ro(t) u'(t),u'(t))+\langle \A_0 (t)u(t),u(t)\rangle\right) &= (\ro'(t)u',u'(t))+\langle \A_0'(t)u(t),u(t)\rangle\\[2\jot]
    &\quad +2(\A_1(t) u(t),u'(t))+(\FF(u_2(t))-\FF(u_1(t)),u'(t))\\[2\jot]
    &\lesssim_{\rho,\A,F,\alpha} (\ro u',u')+\langle\A_0 u, u \rangle.
\end{align*}
Applying the Gronwall lemma, combined with the fact that $u(0)=u'(0)=0$, gives $u \equiv 0$

%\vfill
%\eject

\section{Weak solutions}

After proving the existence of strong solutions in the previous section, we shall weaken the assumptions on the regularity
of coefficients in the wave equation, as well as for  initial conditions and the nonlinear part.
Namely, we assume the following
\begin{equation}\label{eq:assumptions_drugi}
\begin{split}
    &\ro \in \textup{W}^{1,1}(0,T;\L(H)) \text{ and satisfies } (\ref{eq:ro_properties}_{2,3})\\
    &\A_0 \in \textup{W}^{1,1}(0,T;\L(V,V')) \text{ and satisfies } (\ref{eq:a0_properties}_{2,3})\\
    &\A_1 \in \textup{L}^1(0,T;\L(V;H))\\
    &f \in \textup{L}^1(0,T;H)\\
    &g \in \textup{W}^{1,1}(0,T;V')\\
    &F : \R \to \R \text{ is  a continuous function satisfying the sign condition}\\
    &u^0 \in V \text{ and } u^1 \in H.
\end{split}
\end{equation}
Additionally, we shall assume that $\A_0(0)$ is an isomorphism of $V$ and $V'$ and that
\begin{equation}\label{eq:g_iv}
    \GG(u^0) \in \textup{L}^1(\Omega)
\end{equation}
holds, where $\GG$ is defined as in the previous section (Remark \rref{rem:remark_on_coefficients} (c)).

The result we shall prove is precisely stated in the following theorem.

\begin{theorem}\label{thm:existence_drugi}
     Under the above assumptions, there exists a solution $u \in \textup{L}^\infty(0,T;V)$, with $u' \in \textup{L}^\infty(0,T;H)$
     such that $\GG(u) \in \textup{L}^\infty(0,T; \textup{L}^1(\Omega))$,     satisfying the equation 
    \begin{equation}\label{eq:thm_drugi}
            (\ro u')'+ \A u+\FF(u) = f+g \qquad \text{ in\/ } \textup{L}^1(0,T;\textup{L}^1(\Omega)+V'),
    \end{equation}
    with initial conditions
    \[u(0)=u^0, \qquad  u'(0)=u^1.\]
\end{theorem}

\medskip

\begin{remark}\label{rem:initial_conditions}
    Note that such a solution $u$ belongs to the space $\textup{W}^{1,\infty}(0,T;H)$; therefore the initial condition $u(0)=u^0$ makes sense in $H$. On the other hand, from the equation we easily deduce
    \[
        \ro u' \in \textup{W}^{1,1}(0,T;\textup{L}^1(\Omega)+V'),
    \]
    so the second initial condition is to be interpreted in the following sense: $\ro u'$ is absolutely continuous with values in
    $\textup{L}^1(\Omega)+V'$, and the equality $(\ro u')(0)=\ro(0)u^1$ holds in\/ $\textup{L}^1(\Omega)+V'$.
\end{remark}

\bigskip

\subsection{Approximating problems}

The proof of the above theorem relies on approximating input data with more regular ones, so that the results
of second section can be applied. Then, a natural step is to show that these approximating solutions have an accumulation point,
and finally proving, by passing to the limit, that the accumulation point is a solution.

We first state the following approximation lemma for the nonlinear part (the interested reader can find the proof in \cite[Lemma 2.2.]{WS70}). 

\begin{lemma} 
A continuous function $F: \R \to \R$ satisfying the sign condition can be locally uniformly approximated by a sequence of Lipschitz continuous functions also satisfying the sign condition.
\end{lemma}

Denote the members of such a sequence by $F_k$, and analogously as before let us define $G_k$ to be corresponding primitive functions
$G_k(z):=\int_0^z F_k(w)dw$. Then sequences $(F_k)$ and $(G_k)$ satisfy 
properties analogous to \reqref{eq:lip}--\reqref{eq:G_l2_bound} with constants $\lip(F_k)$.

Next, since $u^0$ is not necessarily bounded, we approximate it by a sequence of functions $u^0_j$ given by
\[u^0_j(x)=\xi_j(u^0(x)), \qquad \text{a.e. } x \in \Omega,\]
where $\xi_j$ is defined by
\begin{equation*}
    \xi_j(x):=\begin{cases}
        -j,& x<-j\\
        x,& |x|\leq j\\
        j,& x>j.
    \end{cases}
\end{equation*}
Due to Corollary A.5. of \cite{KS00}, this sequence satisfies
\begin{equation*}
    u^0_j \longrightarrow u^0 \quad   \text{ strongly in } V,
\end{equation*}
as well as
\begin{equation}\label{eq:u0j_properties}
\begin{split}
    &|u^0_j(x)| \leq |u^0(x)| \quad \text{for a.e }x \in \Omega\\[2\jot]
    &u^0_j(x) u^0(x) \geq 0 \quad \text{for a.e }x \in \Omega\\[2\jot]
    &\|u^0_j\|_V \leq \|u^0\|_V \quad j \in \N.
\end{split}
\end{equation}
Additionally, we can extract a subsequence (still dentoted by $u_j^0$) such that we also have convergence almost everywhere
\begin{equation}\label{eq:u0_cvg}
    u^0_j \longrightarrow u^0               \quad  a.e.\: x\in\Omega \:.
\end{equation}
We now proceed to approximate the linear part of the equation in a suitable way. First, let the sequences (indexed by $p\in\N$):
$(\A_{1p})$ in $\textup{W}^{1,1}(0,T;\L(V;H))$, $(u^1_p)$ in $V$ and $(f_p)$ in $\textup{W}^{1,1}(0,T;H)$ satisfy
the following convergences
\begin{equation}\label{eq:cvg_fpgp}
    \begin{array}{rcll}
        %\ro_p & \rightarrow & \ro &\qquad \text{ strongly in } \textup{W}^{1,1}(0,T;\L(H))\\[2\jot]
        %\A_{0p} & \rightarrow \A_{0} \qquad \text{ strongly in } \textup{W}^{1,1}(0,T; \L(V;V'))\\[2\jot]
        \A_{1p} & \rightarrow & \A_1    &\qquad \text{ strongly in } \textup{L}^1(0,T;\L(V;H))\\[2\jot]
        f_p &\rightarrow & f            &\qquad \text{ strongly in } \textup{L}^1(0,T;H)\\[2\jot]
        u^1_p &\rightarrow & u^1        &\qquad \text{ strongly in } H.
    \end{array}
\end{equation}
For the remaining linear terms, we need approximating sequences with some additional properties, stated below.

\begin{lemma}\label{lemma:approx}
    There exist sequences $\A_{0p} \in \textup{W}^{2,1}(0,T;\L(V;V')), \ro_{p} \in \textup{W}^{2,1}(0,T;\L(H))$ and $g_p \in \textup{W}^{2,1}(0,T; V')$ 
    satisfying
    \begin{equation*}
        \begin{split}
        \ro_p & \rightarrow \ro \qquad \text{ strongly in } \textup{W}^{1,1}(0,T;\L(H))\\[2\jot]
        \A_{0p} & \rightarrow \A_{0} \qquad \text{ strongly in } \textup{W}^{1,1}(0,T; \L(V;V'))\\[2\jot]
        g_p & \rightarrow g \qquad \text{ strongly in } \textup{W}^{1,1}(0,T;V'),
        \end{split}
    \end{equation*}
    with
    \begin{equation*}
        \begin{split}
            \ro_{p}(0)&=\ro(0)\\
            \A_{0p}(0)&=\A_0(0)\\
            g_p(0)&=g(0)
        \end{split} \qquad \text{ for each } p \in \N
    \end{equation*}
    and
    \begin{equation*}
        \begin{split}
            \innerh{\ro_p(t)u}{u} \geq \frac{\alpha}{2}\|u\|_H^2, \qquad u \in H\\
            \innerv{\A_{0p}(t)v}{v} \geq \frac{\alpha}{2}\|v\|^2_V, \qquad v \in V
        \end{split}
    \end{equation*}
    independently of both $t \in [0,T]$ and $p \in \N$.
\end{lemma}

\begin{proof}
    Since $\A_{0} \in \textup{W}^{1,1}(0,T;\L(V;V'))$ we may choose a sequence $\Psi_p \in C_c^\infty(0,T;\L(V;V'))$ such that
    \[
        \Psi_p \rightarrow \A_0' \qquad \text{ in } \textup{L}^1(0,T;\L(V;V')).
    \]
    After setting $\A_{0p}(t):=\A_0(0)+\int_0^t \psi_p(s)ds$, we can easily check that $\A_{0p}$ converges strongly to $\A_0$ in $\textup{W}^{1,1}(0,T;\L(V;V'))$ and that $\A_{0p}(0)=\A_0(0)$. Since $\A_{0p}$ also converges to $\A_0$ uniformly on $[0,T]$, we can adjust the sequence if needed so that $\|\A_{0p}-\A_0\|_{\textup{L}^\infty(0,T;\L(V;V'))} < \frac{\alpha}{2}$ for each $p$. Now
    \[
        \innerv{\A_{0p}(t)v}{v} = \innerv{(\A_{0p}(t)-\A_0(t))v}{v}+\innerv{\A_0(t)v}{v} \geq -\frac{\alpha}{2}\|v\|^2_V+\alpha\|v\|_V^2=\frac{\alpha}{2}\|v\|^2_V.
    \]
    Approximating sequences $\ro_p$ and $g_p$ can be constructed in an analogous manner.
\end{proof}

Next, choose a sequence $(u^0_{jp})_p \in V$ such that
\begin{equation}\label{eq:u^0_jp_cvg}
    \begin{split}
        &u^0_{jp} \rightarrow u^0_j \qquad \text{ strongly in } V\\
        &\A_{0}(0)u^0_{jp} - g(0) \in H.
    \end{split}
\end{equation}
Such a sequence can be obtained as follows. As $\A_0(0)$ is an isomorphism of $V$ and $V'$, we have that $\A_0(0)^{-1}H$ is dense in $V$. Hence, we can choose a sequence $(\varphi_p)$ in $\A_0(0)^{-1}H$ such that $\varphi_p \rightarrow u^0_j$ strongly in $V$. Additionally, take a sequence $g_p \in H$ such that $g_p \rightarrow g(0)$ strongly in $V'$. Then the desired sequence is given by $u^0_{jp}:=\varphi_p+\A_0(0)^{-1}(g(0)-g_{p})$.

For each $u^0_{jp},u^1_p,f_p,g_p$ and $F_k$, Theorem \rref{thm:existence_prvi} yields a solution $u_{jpk}$ satisfying
  \begin{center}
        $\begin{cases}
            u_{jpk} \in \textup{L}^\infty(0,T;V)\\
            u_{jpk}'\in \textup{L}^\infty(0,T;V)\\
            u_{jpk}'' \in \textup{L}^2(0,T;H)\\
            \GG_k(u_{jpk}) \in \textup{L}^\infty(0,T; \textup{L}^1(\Omega)),
        \end{cases}$
    \end{center}
of each of the approximate problems
\begin{equation}\label{eq:problem_jk}
    \begin{cases}
    \begin{aligned}
        (\ro_p u_{jpk}')'+\A_p  u_{jpk}+\FF_k(u_{jpk})=&\,f_p+g_p\\
        u_{jpk}(0)=&\,u^0_{jp}\\
        u_{jpk}'(0)=&\,u^1_p,
        \end{aligned}
    \end{cases}
\end{equation}

\subsection{Energy estimates and a solution}

In a similar fashion as in obtaining a priori estimates in the previous section, we first obtain, after denoting
\[
E_{jpk}(t):=\frac{1}{2}(\ro_p u_{jpk}',u_{jpk}')+\frac{1}{2}\langle \A_{0p}u_{jpk},u_{jpk}\rangle + \|\GG_k(u_{jpk})\|_{\textup{L}^1(\Omega)} \;,
\]
the following inequality
\begin{equation*}
   E_{jpk}'(t) \leq \phi_p(t)E_{jpk}(t)+\|f_p(t)\|_H+\langle g_p(t),u_{jpk}'(t)\rangle,
\end{equation*}
where
\[\phi_p(t)=\frac{1}{\alpha/4}\left(\|\ro_p'(t)\|_{\L(H)}+\|\A_{0p}'(t)\|_{\L(V;V')}+\|\A_{1p}(t)\|_{\L(V;H)}+\|f_p(t)\|_H\right).\]
After integrating from $0$ to $t$, where $t\leq T$, we obtain
\begin{equation}\label{eq:energy_jpk}
    E_{jpk}(t) \leq E_{jpk}(0)+\|f_p\|_{\textup{L}^1(H)}+\int_0^t \langle g_p(s),u_{jpk}'(s)\rangle+ \int_0^t \phi_p(s)E_{jpk}(s).
\end{equation}
Since $\|f_p\|$ converges strongly in $\textup{L}^1(0,T;H)$, term $\|f_p\|_{\textup{L}^1(H)}$ is bounded by a constant independent of $j,p$ and $k$. The term involving $g_p$ can be integrated by parts, yielding in the same manner as in \reqref{eq:g_estimates_partial_integration}
\begin{equation*}
    \begin{split}
        \int_0^t \langle g_p(s),u_{jpk}'(s)\rangle ds& = -\int_0^t \langle g_p'(s),u_{jpk}(s)\rangle ds + \langle g_p(t),u_{jpk}(t)\rangle - \langle g_p(0),u_{jpk}(0)\rangle\\[2\jot]
       &\leq 2\kappa\|g_p\|^2_{\textup{W}^{1,1}(V')}+\|g_p'\|_{\textup{L}^1(V')} +\frac{1}{\kappa\alpha}E_{jpk}(t)\\[\jot]
    &\quad +\frac{1}{\kappa\alpha}E_{jpk}(0)+\int_0^t \frac{\|g_p'(s)\|_{V'}}{\alpha/4}E_{jpk}(s)ds.
    \end{split}
\end{equation*}
The first two terms on the right are now bounded due to convergence of $g_p$ in $\textup{W}^{1,1}(0,T;V')$.
The constant $\kappa$ is again chosen so that $\frac{1}{\kappa\alpha} \leq \frac{1}{2}$ holds.

It remains to provide estimates on $E_{jpk}(0)$. Recalling \reqref{eq:u0j_properties}, ($\ref{eq:cvg_fpgp}_1$), Lemma~\rref{lemma:approx} and \reqref{eq:u^0_jp_cvg}, we deduce the boundedness of first two terms in $E_{jpk}(0)$ independent of $j,p,k$. Let us now deal with the term $\|\GG_k(u^0_{jp})\|_{\textup{L}^1(\Omega)}$. Since for each $j \in \N$, $u^0_j$ is a bounded function, $\FF_k(u^0_j)$ converges uniformly in $\Omega$ to $\FF(u^0_j)$ and hence we have
\begin{equation*}
    \begin{split}
        |G_k(u^0_j(x))-G(u^0_j(x))| &\leq \left|\int_0^{u^0_j(x)} F_k(u^0_j(w))-F(u^0_j(w))dw \right|\\[\jot]
        &\leq j \sup_{w \in [-j,j]} |F_k(w)-F(w)| \xrightarrow{k\to\infty} 0, \quad \text{ a.e. } x \in \Omega, 
    \end{split}
\end{equation*}
deducing
\begin{equation*}
    \GG_k(u^0_j) \xrightarrow{k\to\infty} \GG(u^0_j) \quad \text{ in } \textup{L}^\infty(\Omega).
\end{equation*}
As a consequence, since $\Omega$ is bounded, for each $j \in \N$ we have
\[ \|\GG_k(u_j^0) - \GG(u_j^0)\|_{\textup{L}^1(\Omega)} \leq |\Omega|\|\GG_k(u_j^0) - \GG(u_j^0)\|_{\textup{L}^\infty(\Omega)} \xrightarrow{k\to \infty} 0.\]
Hence, we can extract a diagonal subsequence $\GG_{j}(u^0_j):=\GG_{k(j)}(u^0_j)$ of $\GG_k(u^0_j)$ (now denoted simply as $\GG_j$) such that
\[
\int_\Omega |\GG_j(u_j^0) - \GG(u_j^0)| \xrightarrow{} 0.
\]
Furthermore, since $\GG$ is continuous, \reqref{eq:u0_cvg} gives
\[\GG(u_j^0) \longrightarrow \GG(u^0) \text{ a.e.}\]
Now, the sign condition on $\FF$ implies that $\GG$ is nonincreasing on $(-\infty,0)$ and nondecreasing on $(0,\infty)$, which combined with \reqref{eq:u0j_properties} yields $\GG(u^0_j) \leq \GG(u^0)$ a.e. Assumption \reqref{eq:g_iv} therefore allows us to apply the Lebesgue dominated convergence theorem in order to obtain
\[\GG(u^0_j) \longrightarrow \GG(u^0) \quad \text{ in } \textup{L}^1(\Omega).\]
Thus, we conclude
\[
\GG_j(u^0_j) \rightarrow \GG(u^0) \quad \text{strongly in } \textup{L}^1(\Omega) \text{ as well as a.e.}
\]
Finally, from the elementary inequality valid for each $j \in \N$
\[
|G_j(z_2)-G_j(z_1)| \leq \lip(F_j)(|z_1|+|z_2|)|z_2-z_1|,
\]
we have for each $p \in \N$
\[
\|\GG_j(u^0_{jp})-\GG_j(u^0_j)\|_{\textup{L}^1(\Omega)}
\leq \lip(F_j)\left(\|u^0_{jp}\|_H+\|u^0_{j}\|_H\right)\|u^0_{jp}-u^0_{j}\|_H.
\]
Therefore, for each $j \in \N$
\[
\GG_j(u^0_{jp}) \xrightarrow{p \to \infty} \GG_j(u^0_j) \quad \text{ in } \textup{L}^1(\Omega).
\]
We can then once again choose a subsequence of $u^0_{jp}$ denoted by $u^0_{jj}$ such that
\begin{equation*}        %\label{eq:cvg_int_g}
     \GG_j(u_{jj}^0) \longrightarrow  \GG(u^0) \quad \text{ in } \textup{L}^1(\Omega).
\end{equation*}
As a consequence, we also get uniform boundedness of term $\|\GG_j(u^0_j)\|_{\textup{L}^1(\Omega)}$.
Let us now denote the (diagonal) subsequence of $u^0_{jp}$ obtained above with $u_k^0:=u^0_{kk}$ and, accordingly, denote by $u_k$ the sequence corresponding to $u_{kk}$. Taking previous remarks into consideration, upon returning to \reqref{eq:energy_jpk}, we obtain the inequality of the form
\[E_k(t) \lesssim 1+\int_0^t (\phi_k(s)+\|g'(s)\|_{V'})E_k(s)ds,\]
and a simple application of the Gronwall's lemma yields the uniform bound of $E_k$. Thus, we have
\begin{equation}\label{eq:k_bound}\begin{aligned}
    (u_k) \text{ is bounded in }& \textup{L}^\infty(0,T;V)\\[2\jot]
    (u_k') \text{ is bounded in }& \textup{L}^\infty(0,T;H)\\[2\jot]
    \GG_k(u_k) \text{ is bounded in }& \textup{L}^\infty(0,T;\textup{L}^1(\Omega)) \;.
\end{aligned}
\end{equation}
Once again, we can extract a subsequence such that 
\begin{equation}\label{eq:cvg_u_k}
\begin{aligned}
    u_k & \stackrel{\ast}{\rightharpoonup}  u \quad &\text{ in } \textup{L}^\infty (0,T;V)\\
    u_k' & \stackrel{\ast}{\rightharpoonup}  u' \quad &\text{ in } \textup{L}^\infty(0,T;H).
\end{aligned}
\end{equation}

Recall that $u_k$ satisfies
\begin{equation}\label{eq:u_k_equation}
    (\ro_k u_k')'+\A_k u_k + \FF_k(u_k)=f_k+g_k \qquad \text{in } \textup{L}^2(0,T;V').
\end{equation}

We now wish to test this equation against arbitrary $\psi \in \D((0,T)\times \Omega)$ and pass to the limit $k \to \infty$. First we note that the strong convergence $\ro_k \rightarrow \ro$ in $\textup{W}^{1,1}(0,T;\L(H))$ and aforementioned weak-$^*$ convergence of $u_k'$ imply
\begin{equation}\label{eq:cvg_ro_k}
    \ro_ku_k' \xrightharpoonup{\ast} \ro u' \qquad \text{ in } \textup{L}^\infty(0,T;H)   \;.
\end{equation}
Similarly we deduce
\begin{equation}\label{eq:cvg_a_k}
    \begin{split}
         \A_{0k}u_k &\xrightharpoonup{\ast} \A_0u \qquad \text{ in } \textup{L}^\infty(0.T;V')\\[2\jot]
          \A_{1k}u_k &\xrightharpoonup{\ast} \A_1u \qquad \text{ in } \textup{L}^1(0,T;H).
    \end{split}
\end{equation}

It remains to check the convergence of term $\FF_k(u_k)$ in some sense. For this, we refer to \cite[Theorem 1.1]{WS70}, which reads

\begin{theorem}
    Let $\tilde{\Omega}$ be a finite measure space and $X$ and $Y$ Banach spaces. Let $(u_k)$ be a sequence of strongly measurable functions from $\tilde{\Omega}$ to $X$. Let $(\FF_k)$ be a sequence of functions from $X$ to $Y$ such that
    \begin{itemize}
        \item [(a)] $(F_k)$ is uniformly bounded in $Y$ on $B$ for any bounded subset $B$ of $X$,
        \item [(b)] $F_k(u_k(x))$ is strongly measurable and
        \[\sup_k \int_{\tilde{\Omega}} \|u_k(x)\|_X\|F_k(u_k(x))\|_Y dx < \infty,\]
        \item [(c)] $\|F_k(u_k(x))-v(x)\|_Y \rightarrow 0$ a.e.
    \end{itemize}
    Then $v \in \textup{L}^1(\tilde{\Omega};Y)$, and 
    \[\|F_k(u_k)-v\|_{\textup{L}^1(\tilde{\Omega},Y)} \rightarrow 0.\]
\end{theorem}

We intend to apply this theorem to sequences $(u_k)$ and $(F_k)$ with $\tilde{\Omega}=[0,T] \times \Omega$, $X=Y=\R$.
In order to do so, we need to find a uniform bound (without the absolute value, owing to the sign condition) of the term
\[
    \int_0^T (\FF_k(u_k),u_k)dt \:,
\]
and the corresponding almost-everywhere convergence.

Recalling the equation in \reqref{eq:problem_jk}, multiplying it by $u_k$ and integrating over $[0,T]$ we obtain
\begin{equation*}
    \int_0^T (\FF_k(u_k),u_k)dt= \int_0^T (f,u_k) - \langle (\ro u_k')',u_k\rangle - \langle \A_0u_k,u_k \rangle - \langle \A_1 u_k,u_k \rangle\, dt.
\end{equation*}
After integrating by parts in the term involving $\ro$ we get
\begin{equation*}\label{eq:fkuk}
    \begin{split}
    \int_0^T (\FF_k(u_k),u_k)dt= & \int_0^T (f_k,u_k) + \langle g_k,u_k \rangle dt\\[\jot]
    & -\int_0^T(\ro_k u_k',u_k') + \langle \A_{0k}u_k,u_k \rangle + \langle \A_{1k} u_k,u_k \rangle\, dt \\[\jot]
    &+ (\ro_k u_k'(0),u_k(0))-(\ro_k u_k'(T),u_k(T)),
\end{split}
\end{equation*}
% it follows
% \begin{equation*}
% \begin{split}
%     \int_0^T (\FF_k(u_k),u_k)dt \leq \,&\|f\|_{\textup{L}^1(H)}(1+\|u_k\|^2_{\textup{L}^\infty(H)})+\|g\|_{\textup{L}^2}\|u_k\|_{\textup{L}^\infty(H)}\\[\jot]
%     &+\|\ro\|_{\textup{L}^\infty(\L(H))}\|u_k'\|^2_{\textup{L}^2(H)}+\|\A_0\|_{\textup{L}^\infty(\L(V;V'))}\|u_k\|_{\textup{L}^2(V)}^2\\[\jot]
%     &+\|\A_1\|_{\textup{L}^\infty(\L(V;H))}\|u_k\|_{\textup{L}^2(V)}\|u_k\|_{\textup{L}^2(H)} \\[\jot]
%     &+2\|\ro u_k'\|_{\textup{L}^\infty(H)
%     }\|u_k\|_{\textup{L}^\infty(H)}.
% \end{split}
% \end{equation*}
from where the uniform bound easily follows. Since $\FF_k$ converges to $\FF$ pointwise, we have
\[\FF_k(u_k)\rightarrow \FF(u_k) \qquad \text{a.e. in } (0,T)\times \Omega.\]
Because of compact injection $V \xhookrightarrow{} H$, an application of the Aubin–Lions lemma 
yields the strong convergence
\[u_k \rightarrow u, \quad \text{ in } \textup{L}^2((0,T)\times\Omega),\]
which, after possibly extracting another subsequence, yields
\begin{equation*}\label{eq:ae_cvg}
    u_k \rightarrow u \quad \text{a.e. in } (0,T) \times \Omega.
\end{equation*}
Since $F$ is continuous, this also implies
\[\FF(u_k) \rightarrow \FF(u) \quad \text{a.e. in } (0,T) \times \Omega. \]
Therefore, we conclude that
\begin{equation}\label{eq:cvg_fkuk_ae}
    \FF_k(u_k) \rightarrow \FF(u) \quad \text{a.e. in } (0,T) \times \Omega.
\end{equation}
We are now in a position to apply the aforementioned theorem to conclude
\begin{equation}\label{eq:cvg_fkuk}
    \FF_k(u_k) \longrightarrow \FF(u) \text{ strongly in } \textup{L}^1((0,T) \times \Omega).
\end{equation}
Now take $\varphi \in \D(0,T)$ and $v \in \textup{L}^\infty(\Omega) \cap V$ and multiply (in the sense of duality between $\textup{L}^2(0,T;V')$ and $\textup{L}^2(0,T;V))$ the equation \reqref{eq:u_k_equation} with $\varphi \boxtimes v$. Using convergences \reqref{eq:cvg_u_k}, \reqref{eq:cvg_ro_k}, \reqref{eq:cvg_a_k}, \reqref{eq:cvg_fkuk} and \reqref{eq:cvg_fpgp} we obtain
\begin{equation*}
    \int_0^T -(\ro u',v)\varphi' + \langle \A_0 u,v\rangle \varphi + (\A_1 u, v)\varphi + (\FF(u),v)\varphi \; dt = \int_0^T (f,v)\varphi+\langle g,v\rangle \varphi.
\end{equation*}
Since $v \in \textup{L}^\infty(\Omega) \cap V$ was arbitrary, we conclude
\begin{equation*}
    (\ro u')'=f+g-\A u-\FF(u) \qquad \text{in } \D'(0,T; \textup{L}^1(\Omega)+V')
\end{equation*}
and therefore,
\begin{equation*}   %\label{eq:ro_u'_space}
    (\ro u')' \in \textup{L}^1(0,T;\textup{L}^1(\Omega)+V'),
\end{equation*}
with the equation
\begin{equation*}
    (\ro u')'+\A u+\FF(u)=f+g
\end{equation*}
now valid in the sense of $\textup{L}^1(0,T;\textup{L}^1(\Omega)+V').$

Furthermore, as in \reqref{eq:cvg_fkuk_ae} we have
\[
    \GG_k(u_k) \rightarrow \GG(u) \qquad \text{a.e. in } (0,T)\times \Omega,
\]
and an application of the Fatou's lemma yields together with ($\ref{eq:k_bound}_3$) 
\[
    \|\GG(u)\|_{\textup{L}^\infty(0,T;\textup{L}^1(\Omega))} \leq \liminf_k{\|\GG_k(u_k)\|_{\textup{L}^\infty(0,T;\textup{L}^1(\Omega))}} \lesssim 1.
\]

Initial conditions in the sense of Remark \rref{rem:initial_conditions} are now verified in the standard manner.

\bigskip

\subsection{Uniqueness and some improvements in a special case}

Some improvements on the solution, including the uniqueness, can be obtained by imposing additional conditions.

\begin{theorem}\label{prop:uniquness_cont_nonlinear}
    In the case $d \geq 3$, assume additionally that $\A_1$ satisfies
\begin{equation}   \label{eq:a1_uniq}
    \A_1 \in \textup{L}^1(0,T;\L(H,V')),
\end{equation}
    and $F$ is a function of class $\textup{C}^1$ satisfying growth condition
\begin{equation*}       %\label{eq:growth}
    |F'(z)| \leq C|z|^p, 
\end{equation*}
    for some $\frac{1}{d} \leq p \leq \frac{2}{d-2}$. Then the solution of \reqref{eq:thm_drugi} is unique and moreover, the equality
    \[
         (\ro u')'+\A u+\FF(u)=f+g
    \]
    holds in the sense of $\textup{L}^1(0,T;V')$.
\end{theorem}

\begin{proof}
The improvement on the space in which the equation holds can be obtained by repeating the steps of the proof of the Theorem \rref{thm:existence_drugi}; we shall just give a brief sketch of required adjustments. 

In this case, the approximation of the nonlinearity is not needed; neither approximating $u^0$ with a sequence of bounded functions is required.
After an approximation of the linear part and initial conditions is completed (in particular, we only need sequences indexed by one parameter, hereafter $k$), we obtain the unique solution $u_k \in \textup{L}^\infty(0,T;V) \cap \textup{W}^{1,\infty}(0,T;H)$ of
\[
    \begin{cases}
    \begin{aligned}
        (\ro_k u_{k}')'+\A_k  u_{k}+\FF(u_{k})=&\,f_k+g_k\\
        u_{k}(0)=&\,u^0_{k}\\
        u_{k}'(0)=&\,u^1_k,
        \end{aligned}
    \end{cases}
\]
After denoting the energy at time $t$
\[
    E_k(t):=\frac{1}{2}\innerh{\ro_k(t)u_k(t)}{u_k(t)}+\frac{1}{2}\innerv{\A_{0k}(t)u_k(t)}{u_k(t)}+\|\GG(u_k(t))\|_{\textup{L}^1(\Omega)},
\]
the energy estimates and an accumulation point $u \in \textup{L}^\infty(0,T;V)\cap \textup{W}^{1,\infty}(0,T;H)$ are obtained in an
analogous fashion, taking into account the following: The growth condition, together with the sign condition imply
\[
    |F(z)| \lesssim 1+|z|^{\frac{d}{d-2}}, \qquad z \in \R.
\]
Therefore, for any $v \in V$ we have 
\[
    \|\FF(v)\|_H \lesssim 1+\|v\|_{\textup{L}^{\frac{2d}{d-2}}(\Omega)} \lesssim 1+ \|v\|_V,
\]
so $\FF(v) \in H$. Similarly, $\GG(v) \in \textup{L}^1(\Omega)$. The final difference lies in the examination of the convergence of the term $\FF(u_k)$. Here we apply Lemma 1.3 from \cite{L69} to that sequence and obtain
\[
    \FF(u_k) \xrightharpoonup{} \FF(u) \qquad \text{ weakly in } \textup{L}^2(0,T;H).
\]
It is then straightforward to verify that $u$ satisfies
\[
    (\ro u')'+\A u + \FF(u)=f+g \qquad \text{ in the sense of } \textup{L}^1(0,T;V').
\]

For uniqueness, assume $u_1,u_2$ are two solutions and denote $u=u_1-u_2$. Since there is not enough regularity in $u'$ to use it as a test function, we circumvent this problem using the standard argument. Fix $0 < s < T$ and define
\[
    v(t):=\begin{cases}
    -\int_s^t u(\tau)d\tau, & 0 \leq t \leq s,\\
    0, & s < t \leq T.
\end{cases}
\]
Then $v \in \textup{W}^{1,\infty}(0,T;V)$, so we get
\[
\int_0^s \langle (\ro u')',v \rangle + \langle \A_0 u,v\rangle+(\A_1u,v)+(\FF(u_1)-\FF(u_2),v) \, dt =0.
\]
Since $(\ro u')(0)=v(s)=0$, integrating by parts in the first term yields
\[
    \int_0^s -( \ro u',v' ) + \langle \A_0 u,v \rangle+(\A_1u,v)+(\FF(u_1)-\FF(u_2),v) \, dt =0.
\]
As $v'=-u$ for $0 \leq t \leq s$, we can write
\[
    \int_0^s (\ro u',u)-\langle \A_0 v',v \rangle+( \A_1u,v) +(\FF(u_1)-\FF(u_2),v) \, dt=0.
\]
It follows that
\begin{align*}
    (\ro(s)u(s),u(s))+\langle \A_0(0)v(0),v(0)\rangle =& \int_0^s \frac{d}{dt}\frac{1}{2}\Big((\ro u,u)-\langle \A_0 v,v \rangle \Big) \, dt\\ 
    = &\int_0^s  -(\ro'u,u)-\langle \A_0' v,v\rangle+(\A_1u,v) -(\FF(u_1)-\FF(u_2),v)\, dt.
\end{align*}
The left hand side can then be estimated from below:
\begin{equation*}
    \|u(s)\|_H^2+\|v(0)\|_V^2 \leq \frac{1}{\alpha}\left((\ro(s)u(s),u(s))+\langle \A_0(0)v(0),v(0)\rangle\right).
\end{equation*}
For the right hand side, we first deal with the nonlinear part. The growth condition  and the Mean-value theorem yield the inequality
    \[|\FF(z_1)-\FF(z_2)| \leq C\max\{|z_1|^p,|z_2|^p\}|z_1-z_2|.\]
    Therefore, the extended Hőlder inequality applied for $\frac{1}{q}+\frac{1}{d}+\frac{1}{2}=1$ gives
    \begin{align*}
       \int_0^s (\FF(u_1)-\FF(u_2),v) dt \, \lesssim &\int_0^s (\max\{|u_1|^{p},|u_2|^p\}|u_1-u_2|,|v|) dt\\
       \lesssim & \int_0^s \left((|u_1|^p+|u_2|^p)|u_1-u_2|,|v|\right)dt\\
       \lesssim & \int_0^s \left\| |u_1(t)|^p+|u_2(t)|^p \right\|_{\textup{L}^{d}(\Omega)}\,\|u(t)\|_{\textup{L}^q(\Omega)}\,\|v(t)\|_H \,dt.
    \end{align*}
    Since $q> 2$, we have $\|u(t)\|_{\textup{L}^q(\Omega)} \leq \|u(t)\|_H$. Next, since $1 \leq dp\leq q=\frac{2d}{d-2}$, the Sobolev embedding theorem gives
    \begin{align*}
         \| |u_1(t)|^p+|u_2(t)|^p \|_{\textup{L}^d(\Omega)} \,\leq & \,\| |u_1(t)|^p \|_{\textup{L}^d(\Omega)}+\||u_2(t)|^p\|_{\textup{L}^d(\Omega)}\\[2\jot]
         \lesssim& \|u_1(t)\|_V^p+\|u_2(t)\|_V^p\\[2\jot]
         \lesssim& \|u_1\|^p_{\textup{L}^\infty(0,T;V)}+\|u_2\|^p_{\textup{L}^\infty(0,T;V)}.
    \end{align*}
    Finally, we obtain
    \begin{align*}
        \int_0^s &\| |u_1(t)|^p+|u_2(t)|^p \|_{\textup{L}^{d}(\Omega)}\,\|u(t)\|_{\textup{L}^q(\Omega)}\,\|v(t)\|_H \,dt \\
        &\lesssim \left(\|u_1\|^p_{\textup{L}^\infty(0,T;V)}+\|u_2\|^p_{\textup{L}^\infty(0,T;V)}\right)\int_0^s \|u(t)\|_H\|v(t)\|_H\,dt\\
        &\lesssim \left(\|u_1\|^p_{\textup{L}^\infty(0,T;V)}+\|u_2\|^p_{\textup{L}^\infty(0,T;V)}\right)\int_0^s \|u(t)\|_H^2+\|v(t)\|_V^2\,dt. 
    \end{align*}
Recalling \reqref{eq:a1_uniq}, we further estimate
\begin{align*}
    \int_0^s  &-(\ro'u,u)-\langle \A_0' v,v\rangle+(\A_1u,v) \\
    &\leq \int_0^s \|\ro'(t)\|_{\L(H)}\|u\|_H^2+\|\A_0'(t)\|_{\L(V;V')}\|v(t)\|_V^2+\|\A_1(t)\|_{\L(V;H)}\|u\|_H\|v\|_V \,dt\\
    &\leq \int_0^s \left(\|\ro'(t)\|_{\L(H)}+\|\A_0(t)\|_{\L(V;V')}+\|\A_1(t)\|_{\L(V;H)}\right)\left(\|u(t)\|_H^2+\|v(t)\|^2_V\right) \,dt \;.
\end{align*}
Combining it all together yields (for some constant $C>0$)
\begin{align}\label{eq:10}
    \|u(s)\|_H^2+\|v(0)\|_V^2 \leq C\int_0^s \beta(t)\left(\|u(t)\|_H^2+\|v(t)\|_V^2\right) \,dt, 
\end{align}
where $\beta$ denotes the (positive) function
\[
    \beta(t)=\|\ro'(t)\|_{\L(H)}+\|\A_0(t)\|_{\L(V;V')}+\|\A_1(t)\|_{\L(V;H)}+\lip (F)\;.
\]
Note that we have $\beta \in \textup{L}^1(0,T)$.

For $0\leq t \leq T$ define 
$$w(t):=\int_0^t u(\tau)d\tau \;.$$
The inequality in \reqref{eq:10} then becomes
\begin{equation*}           %\label{eq:11}
    \|u(s)\|_H^2+\|w(s)\|_V^2 \leq C\int_0^s \beta(t)\left(\|u(t)\|_H^2+\|w(s)-w(t)\|_V^2\right)\, dt.
\end{equation*}
However
$$
\|w(t)-w(s)\|_V^2 \leq 2(\|w(t)\|_V^2+\|w(s)\|_V^2) \:,
$$
so continuing the above estimate we get
\begin{equation*}
    \|u(s)\|_H^2+(1-2C\|\beta\|_{\textup{L}^1(0,s)})\|w(s)\|_V^2 \leq C \int_0^s \beta(t)\left(\|u(t)\|_H^2 + \|w(t)\|_V^2\right)\, dt.
\end{equation*}
Choose $T_1$ such that $\left(1-2C\|\beta\|_{\textup{L}^1(0,T_1)}\right) \geq \frac{1}{2}$. Then for each $0 \leq s \leq T_1$ we have
\begin{equation*}
    \|u(s)\|_H^2+\|w(s)\|_V^2 \leq C \int_0^s \beta(t)\left(\|u(t)\|_H^2 + \|w(t)\|_V^2\right)\, dt .
\end{equation*}
Applying Gronwall's inequality we get $u \equiv 0$ on $[0,T_1]$.
By using the same argument on intervals $[T_1,2T_1]$, $[2T_1,3T_1]$, etc.~we eventually get $u \equiv 0$ on $[0,T]$. 
\end{proof}

\section{Some examples and extensions}

To illustrate the results of the previous section on a concrete example, let us consider the problem
\begin{equation}       \label{exa:eq}
(\rho u')'-\div(\textbf{A}\nabla u)+\mathbf{b}\cdot \nabla u + cu +u|u|^{p}=f+g
\end{equation}
with prescribed initial conditions
\[
    u(0)=u^0, \qquad u'(0)=u^1.
\]
In order for this problem to fit within the established framework, we impose the following assumptions:
\begin{itemize}
\item Function $\rho$ belongs to the space $\textup{W}^{1,1}(0,T;\textup{L}^\infty(\Omega))$ and it satisfies $\rho \geq \alpha >0$ for a.e. on $(0,T)\times \Omega$.
\item We write $\textbf{A}$ as the sum of its symmetric and antisymmetric parts,
\[
    \mathbf{A_s}=\frac{1}{2}(\mathbf{A}+\mathbf{A}^T), \qquad 
    \mathbf{A_a}=\frac{1}{2}(\mathbf{A}-\mathbf{A}^T).
\]
We assume that the symmetric part satisfies
\[
    \mathbf{A_s} \in \textup{W}^{1,1}(0,T; \textup{L}^\infty(\Omega; \R^{d\times d})) 
    \quad \text{and} \quad 
    \mathbf{A_s}(t,x) \geq \alpha \mathbf{I}.
\]
For the antisymmetric part, we assume that
\[
    \div \mathbf{A_a} \in \textup{L}^1((0,T)\times \Omega; \R^{d}),
\]
where $\div \mathbf{A_a}$ is understood as the vector in $\R^d$ whose entries are the (standard) divergences of the corresponding columns of $\mathbf{A_a}$.

\item The coefficients in the lower-order terms satisfy
\[
    \mathbf{b} \in \textup{L}^{1}((0,T)\times \Omega;\R^d), 
    \quad 
    c \in \textup{L}^{1}((0,T)\times \Omega).
\]

\item The right-hand side of the equation is given by
\[
    f \in \textup{L}^{1}(0,T;H), 
    \qquad 
    g \in \textup{W}^{1,1}(0,T;V').
\]

\item The initial conditions satisfy $u^0 \in V$ and $u^1 \in H$.
\end{itemize}
We may now define appropriate operators $\ro$ and $\A$. Clearly, $\ro$ corresponds to multiplication by $\rho$, namely
\[
    \ro(t)v = \rho(t)v, \qquad v \in H.
\]
Under our assumptions, $\rho$ satisfies ($\ref{eq:assumptions_drugi}_1$). Next, we identify the operator $\A$ as
\[
    \A(t)v = -\div(\mathbf{A}(t)\nabla v)+\mathbf{b}(t)\cdot \nabla v + cv.
\]
Using the decomposition $\mathbf{A}=\mathbf{A_s}+\mathbf{A_a}$, we further define, for $v \in V$,
\begin{equation*}
    \begin{split}
        \A_0(t)v&=-\div(\mathbf{A_s}\nabla v),\\
        \A_1(t)v&=-\div(\mathbf{A_a} \nabla v)+\mathbf{b}\cdot \nabla v + cv.
    \end{split}
\end{equation*}
Once again, it follows immediately from the preceding assumptions that $\A_0$ satisfies ($\ref{eq:assumptions_drugi}_2$).

Additionally, the part of the operator $\A_1$ containing the terms with $\mathbf{b}$ and $c$ belongs to $\textup{W}^{1,1}(0,T;\L(V,H))$. To verify that the same holds for the divergent part, we take $u \in \D(\Omega)$ and $v \in H$ to obtain
\begin{equation*}
    \begin{split}
        |(-\div (\mathbf{A_a}(t)\nabla u), v)| 
        &\leq |((-\div \mathbf{A_a}(t)) \cdot \nabla u, v)|
        +| \langle {\mathbf{A_a}(t) \cdot \nabla^2 u},v\rangle|\\
        &= |((-\div \mathbf{A_a}(t)) \cdot \nabla u, v)|\\[4\jot]
        &\leq \|\div \mathbf{A_a}(t)\|_{\textup{L}^\infty}\|u\|_V\|v\|_H,
    \end{split}
\end{equation*}
which shows that $-\div(\mathbf{A_a} \nabla \cdot \;)$ belongs to $\textup{L}^1(0,T;\L(V,H))$.
(In the equality above we have just noted that the dot product of antisymmetric $\mathbf{A_a}$ and symmetric Hessian is zero.)
Moreover, same argument shows that it also belongs to the space $\textup{L}^1(0,T;\L(H,V'))$. Since the same obviously holds for $v \mapsto \mathbf{b}\cdot \nabla v + cv$, we also conclude $\A_1 \in L^1(0,T;\L(H,V'))$.

Finally, in this case the nonlinear part is given by the smooth function $F(z)=z|z|^p$, for which the minimal assumption is $p>0$. Note that its primitive is then
\[
    G(z)=\frac{1}{p+2}|z|^{p+2}.
\]
As such, $F$ satisfies the sign condition and
\[
    |F'(z)|=(p+1)|z|^p.
\]
If no further assumptions on $p$ are imposed, it suffices to require $u^0 \in \textup{L}^{p+2}(\Omega)$. Theorem \ref{thm:existence_drugi} then applies, yielding the existence of a function
\[
    u \in \textup{L}^\infty(0,T;V \cap \textup{L}^{p+2}(\Omega)), 
    \qquad 
    u' \in \textup{L}^\infty(0,T;H),
\]
such that \eqref{exa:eq} holds in the sense of $\textup{L}^1(0,T; V'+\textup{L}^1(\Omega))$.

If we additionally assume $\frac{1}{d} \leq p \leq \frac{2}{d-2}$, then the solution is unique, and the equation holds in the sense of $\textup{L}^1(0,T;V')$.

\section{Acknowledgments}

Matko Grbac has been supported in part by Croatian Science Foundation under the project IP-2022-10-7261.
\vfill
%\eject

%Additional references  

%\cite{AM07, HB73, HB11, BM10, CD11, LE90, GT01, MMP2008, SM73, WS66, LT97, LT01, LT02, LT17, LT23, JW87}

%%       \citet{<label>} ==> Lamport (1994)
%%       \citep{<label>} ==> (Lamport, 1994)
%%


%\bibliographystyle{elsarticle-num}
\bibliography{elsarticle/bibliography.bib}




\end{document}